\documentclass[11pt,a4paper]{article}
\usepackage{dsfont}
\usepackage{amsmath}
\usepackage{amssymb}
\usepackage[applemac]{inputenc}
\usepackage{graphicx}
\usepackage{color}
\newcommand{\EE}{\mathbb{E}}

\newcommand{\PP}{ \mathbb{P}}
\newcommand{\FF}{ \mathbb{F}}

\newcommand{\M}{{\mathbb M}}

\DeclareMathOperator*{\esssup}{ess\,sup}

\newcommand{\ds}{\displaystyle}

\newcommand{\ito}{\mathcal{I}^{n}}
\newcommand{\control}{(L^{2,2}_{\mathbb{F}})^{m}}

\newenvironment{proof}[1][Proof]{\textbf{#1.} }{\ \rule{0.5em}{0.5em}}
\newtheorem{remark}{\textbf{Remark}}[section]
\newtheorem{lemma}{\textbf{Lemma}}[section]
\newtheorem{theorem}{\textbf{Theorem}}[section]
\newtheorem{corollary}{\textbf{Corollary}}[section]
\newtheorem{proposition}{\textbf{Proposition}}[section]

\newtheorem{definition}{\textbf{Definition}}[section]

\numberwithin{equation}{section}
\title{ Some sensitivity results in stochastic optimal control:  A Lagrange multiplier point of view \author{
J. Backhoff  \thanks{ Berlin Mathematical School and Institut f\"ur Mathematik of Humboldt Universit\"at zu Berlin.} \and     F. J. Silva  \thanks{ Institut de recherche XLIM-DMI, UMR-CNRS 7252 Faculté des sciences et techniques 
Université de Limoges, 87060 Limoges, France (francisco.silva@unilim.fr).} \thanks{ This work was initiated during a one month visit to the Hausdorff Research Institute for Mathematics at the University of Bonn in the framework of the Trimester Program Stochastic Dynamics in Economics and Finance}\thanks{The authors express their gratitude to J. Fontbona for   useful discussions regarding the content of Section 3.}}}

\def\dd{{\rm d}}

\def\weight(#1,#2){c_{#1,#2}}

 \if {
  
 } \fi

\if {

}{{\caption}}
} \fi

\def\F{\mathcal{F}}
\def\G{\mathcal{G}}

\def\I{\mathcal{I}}

\def\L{\mathcal{L}}
\def\M{\mathcal{M}}

\def\P{\mathcal{P}}

\def\U{\mathcal{U}}

\def\eps{\varepsilon}

\def\half{\mbox{$\frac{1}{2}$}}

\def\1B{{\bf  1}}

\newcommand{\RR}{\mathbb{R}}

\newcommand\be{\begin{equation}}
\newcommand\ee{\end{equation}}
\newcommand\ba{\begin{array}}
\newcommand\ea{\end{array}}
\newcommand{\bean}{\begin{eqnarray*}}
\newcommand{\eean}{\end{eqnarray*}}

\def\ds{\displaystyle}

\begin{document}
\maketitle
\abstract In this work we provide a first order sensitivity analysis of some parameterized stochastic optimal control problems. The parameters can be given by random processes.
The main tool is the  one-to-one correspondence between the adjoint states appearing in a weak form of the stochastic Pontryagin principle and the Lagrange multipliers associated to the state equation. \medskip
\normalsize
\begin{description}
\item [{\bf MSC 2000}: 93E20,49Q12,47J30,49N10,91G10] .
\end{description}
\section{Introduction}
One of the most important results in stochastic optimal control theory is Pontryagin Principle, introduced and refined by \cite{Kus65}, \cite{Bismut76}, \cite{haussmann1986stochastic}, \cite{Ben83} and \cite{Peng90} among others (see \cite[Chapter 3, Section 7]{YongZhou} for a historical account).  In its simplest form, it says that almost surely the optimal control minimizes an associated {\it Hamiltonian}. This Hamiltonian depends on the optimal state and an {\it adjoint pair}, which solves an associated Backward Stochastic Differential Equation (BSDE for short).  
 Roughly speaking, the mentioned necessary condition appears as one perturbs the optimal control and analyzes up to first order (or second-order, if the volatility term is controlled and the set of admissible controls is non-convex) the impact of such perturbation on the cost function. A natural question that arises is whether by regarding the stochastic optimal control problem as an infinite dimensional optimization problem in an appropriate functional setting, the usual machinery of   optimization theory yields an interpretation of the aforementioned adjoint states. From this perspective it is conceivable that fundamental tools such as convex-duality, Lagrange multipliers and non-smooth analysis  (to name a few) may shed new lights and provide new interpretations into the field of stochastic optimal control.

The idea of dealing with stochastic optimal control problems from the point of view of optimization theory is not new. In a remarkable article \cite{Bismut73*}, the author extends to the stochastic case the results of \cite{Rockafellar70} obtained in the deterministic framework. For convex problems, he proves essentially that    the solutions of the original optimization problem and its dual, in the sense of convex analysis, must fulfil the conditions appearing in Pontryagin's  Principle. In the non-convex case, a  very  interesting  analysis is performed in \cite{Loewen87} where the author uses non-smooth analysis techniques to tackle the case of a non-linear controlled Stochastic Differential Equation (SDE for short)\footnote{More recently, in e.g. \cite{ChengYan,Kosmol}, the Lagrange multiplier technique has been applied formally in order to derive optimality conditions. However,  no connexions with Pontryagin's principle are analyzed.}.
%


In this article we develop a rigorous   functional framework under which the Lagrangian  approach to stochastic optimal control becomes fruitful. As a matter of fact, we relate the adjoint states appearing in the Pontryagin principle with the Lagrange multipliers of the associate optimization problem,  thus extending the results of \cite{Bismut73*} in the convex case, by using a different method. In several interesting cases, this result allows us to perform a {\it first order  sensitivity analysis of the value function}, under infinite dimensional perturbations of the dynamics.  To the best of our knowledge, this type of sensitivity result had been obtained for finite dimensional perturbations of the initial condition (see the works   \cite{Loewen87,Zhou90,MR1128491}) only.  We restrict ourselves to a finite-horizon, brownian setting, yet consider the case of non-linear controlled SDEs with random coefficients and the control being present both in the drift and diffusion parts, pointwise  convex constraints on the controls, and finite dimensional constraints of expectation-type on the final state. In mathematical language, we deal with problems of the form:
$$\left.\ba{l} \ds \inf_{(x,u)} \EE \left[ \int_{0}^{T}\ell(\omega, t, x(t),u(t)) \dd t + \Phi(\omega,x(T))  \right]\\[12pt]
\mbox{s.t. } \; \; \; \; 	x(t) = x_0+ \int_{0}^{t} f(s,  x(s),u(s)) \dd s+   \int_{0}^{t} \sigma( s, x(s), u(s)) \dd W(s),\;  \forall \; t\in [0,T[,\\[4pt]  
						
\hspace{1.05cm} \EE\left( \Phi_{E}( x(T)) \right)=0, \hspace{0.4cm}  \EE\left( \Phi_{I}( x(T)) \right)\leq 0,  \hspace{0.4cm} u(\omega,t) \in U \; \; \mbox{a.s.}\ea \right\}\eqno(CP)
$$
where $\ell$, $\Phi$, $f$, $\sigma$, $x_0$, $\Phi_{E}$, $\Phi_{I}$ are the data of the problem, which can be random, satisfying some natural assumptions detailed in Section \ref{mwmnnsndnadawwwapapapapas}, and $U\subseteq \RR^{m}$ is a convex set. Under some standard assumptions, we have that for every square integrable and progressively measurable control $u$, there exists a unique solution $x[u]$ of the SDE in $(CP)$. In this sense, problem $(CP)$ can be reformulated in terms of $u$ only and the SDE constraint can be eliminated. However, we have chosen to work with the pair $(x,u)$ and keep the SDE constraint in order to associate to it a Lagrange multiplier, in view of the important consequences of this approach in the sensitivity analysis of the optimal cost of $(CP)$  (see Section \ref{sens}).

By defining a Hilbert space topology on the space of It\^o processes, we naturally deduce that whenever the Lagrange multipliers  associated to the SDE constraint in $(CP)$ exists they must be It\^o processes themselves. With this methodology we can prove a one-to-one simple relationship between the aforementioned Lagrange multipliers and the adjoint states appearing in a weak form of Pontryagin's principle. More concretely, we say that $(p,q)$ is a {\it weak-Pontryagin multiplier}  at a solution $(x,u)$ if the same conditions appearing in the usual Pontryagin principle holds true (see \cite[Theorem 3]{Peng90}), except for the condition of minimization of the Hamiltonian  
which is replaced by the  weaker statement corresponding to its first order optimality condition (see Section \ref{descripcionpmp} for a detailed exposition). Thus, it is easily seen that every adjoint pair appearing in the usual Pontryagin principle is a weak-Pontryagin multiplier. In Theorem \ref{Teoidentificacion} we  prove that given a weak-Pontryagin multiplier $(p,q)$, the process  
\be\label{lamultiplier} \lambda(\cdot):= p(0)+ \int_{0}^{\cdot} p(s) \dd s + \int_{0}^{\cdot} q(s) \dd W(s), \ee
is a Lagrange multiplier associated to the SDE constraint in $(CP)$. Conversely, every Lagrange multiplier $\lambda(\cdot)= \lambda_{0} + \int_{0}^{\cdot} \lambda_{1}(s) \dd s + \int_{0}^{\cdot} \lambda_{2}(s) \dd W(s)$, associated to this constraint, satisfies that $\lambda_{0}= \lambda_{1}(0)$ and $(\lambda_{1}, \lambda_{2})$ is a weak-Pontryagin multiplier. What is more, in the case of {\it convex costs and linear dynamics} we derive in Theorem \ref{sensibilidad1} the existence of Lagrange multipliers and hence the Pontryagin principle,  by solely invoking the theory of Lagrange multipliers in Banach spaces (see e.g.  \cite{bstour,BonSha} for a survey). Even if this type of arguments can be extended to the case of non-convex costs (see Remark \ref{mqmennnnaasasasa}{\rm(iv)}), at the present time we do not know if it is possible by the latter theory to prove Pontryagin's principle in the case of non-linear dynamics.

One advantage of identifying the Lagrange multipliers of an optimization problem is that, under some precise conditions, these multipliers allow to perform a first-order sensitivity analysis of the value function as a function of the problem parameters. In a nutshell, if the optimization problem at hand is convex (this is the case of convex costs and linear equality constraints) or smooth and stable with respect to parameter perturbations (e.g. if the optimizers converge as we vary the parameters, and the functions involved are at least continuously differentiable) then the sensitivity of the value function in terms of the perturbation is  related to the derivative of the Lagrangian with respect to the parameters taken in the perturbation direction (see e.g. \cite[Section 4.3]{BonSha}).

Using the identification of Lagrange and weak-Pontryagin multipliers we establish in Section \ref{sens} our main results:  In Theorem \ref{sensibilidad1} we prove, for example, that for stochastic optimal control problems with convex costs and linear dynamics, an additive (random, time-dependent) perturbation $(\Delta f,\Delta \sigma)$ to the drift and  diffusion parts of the controlled SDE changes the value function (up to first order) by exactly 
$$\EE\left( \int_{0}^{T}p(t)^{\top}  \Delta f(t) \dd t\right) + \EE\left( \int_{0}^{T} \mbox{{\rm tr}}\left[q^{\top}(t)  \Delta \sigma (t) \right] \dd t\right),$$
where $(p,q)$ is (in this case) the unique adjoint state appearing in the Pontryagin's principle. A simple corollary of this is that if one perturbs a deterministic optimal control problem by a small (brownian) noise term, the value function remains unaltered up to first-order, as was observed in \cite{Loewen87} by other methods. 

At the present point we cannot extend the previous sensitivity analysis to non-convex problems. However, we can tackle some cases of non-additive parameter perturbations of convex stochastic optimal control problems. This is an important improvement from what was outlined in the previous paragraph, as in practice parameter error/inaccuracy can propagate in very complicated fashions if for instance this error is amplified by the decision (control) variable. This is the setting we face in two examples we deal with in this article; the stochastic Linear-Quadratic (LQ) control problem and the Mean-Variance portfolio selection problem, which is an LQ problem with a constraint on the expected value of the final state.   In these problems, it is natural to consider  perturbations of the matrices appearing in the dynamics that multiply  either the state or the control.  The main tool here is the stability result in Proposition \ref{qmmqemqmndndndndaaa} regarding a weak continuity property for the solutions of linear SDE and BSDE in terms of the parameters.

As suggested by their name, in a stochastic LQ problem one seeks to minimize a quadratic functional of the state and control variables, which are related through a linear SDE. Such problems are to be found everywhere in engineering and economics  sciences  and we refer the reader to  \cite{Bismut76a, ChenYong01,Tang,YongZhou} and the references therein for an exposition of the theory. Our main results here are a {\it strong stability property} for the solutions of parameterized unconstrained convex LQ problems (see Proposition \ref{convergencia}) and  Theorem \ref{teoremalq}, where we provide a complete sensitivity analysis for the value function in terms of the parameters. More precisely, we prove that the optimal cost depends in a {\it continuously differentiable} manner  on the various parameters and we give explicit expressions for the associated derivatives. From the practical point of view, this result can have interesting applications. As matter of fact, recall that the resolution of deterministic LQ problems can be done through the resolution of an associate deterministic backward Riccati differential equation. The analogous result holds true in the stochastic framework \cite{Tang}, but in that case the Riccati equation is a highly nonlinear BSDE. Therefore, for small random perturbations of the matrices of a deterministic LQ problem, it seems  reasonable to approximate the value function of the perturbed problem as the value of the deterministic one plus a first order term, which can be calculated in terms of the solution of the {\it deterministic} Riccati equation (see Remark \ref{remarkutil}{(i)}).

In the classical Mean-Variance portfolio selection problem, one seeks to find the portfolio rendering the least variance of the terminal wealth with a guaranteed fixed expected return. This is a very central topic in finance and economics, and we refer the reader to \cite{ZhouLi00} (random coefficients), \cite{Oks04} (case with jumps), among others. As for the general LQ case, our  major contributions here are Proposition \ref{lemconv}, dealing with an {\it stability} analysis for the optimal solutions in terms of the perturbation parameters (the initial capital,  deterministic interest/saving rates, the desired return, the drift and the diffusion coefficients) and Theorem \ref{corosens}, where we prove that the optimal cost is {\it $C^{1}$} with respect to those perturbations.

The article is structured as follows. In Section \ref{preliminares} we introduce relevant notation and present the mentioned Hilbert space topology in the space on It\^o processes, along with some needed technical lemmata. Next in Section \ref{BSDE} we identify some operators that will be of importance in the next section and find their adjoints in terms of associated BSDEs. In Section \ref{mwmnnsndnadawwwapapapapas} we define the optimal control problem, we study the differentiability properties of the several functions appearing in the data  and culminate establishing the one-to-one relationship between Lagrange multipliers and weak-Pontryagin multipliers. Then in Section \ref{sens} we take advantage of the Lagrange point of view and analyze the differentiability properties of the value function with respect to its parameters in the case of linear perturbations (Section \ref{sens1}) of convex problems, the case of stochastic Linear-Quadratic problems (Section \ref{sens2}) and Mean-Variance portfolio optimization problem (Section \ref{sens3}).

 \section{Preliminaries and functional framework}\label{preliminares} 
Let  $T>0$ and consider a filtered probability space  $(\Omega, \F, \mathbb{F}, \mathbb{P})$, on which a $d$-dimensional ($d \in \mathbb{N}^{*}$)     Brownian motion $W(\cdot)$  is defined. We suppose that $\mathbb{F}=\left\{\F_{t}\right\}_{0\leq t \leq T}$  is the   natural filtration, augmented by all $\mathbb{P}$-null sets in $\F$, associated to $W(\cdot)$. We recall that $\mathbb{F}$ is right-continuous. Given   $\beta, p\in [1,\infty]$  and $n\in \mathbb{N}$ let us consider the Banach spaces\small
$$\ba{rl}
(L_{\FF}^{\beta,p})^{n}&:= \left\{ v \in L^{\beta}\left(\Omega ; L^{p}\left([0,T]; \RR^{n}\right)\right) ; \ (t,\omega) \to v(t,\omega):=v(\omega)(t) \  \mbox{is } \mathbb{F}\mbox{-progressively measurable}\right\}.
\ea	 
$$\normalsize
We write  $\| \cdot \|_{\beta,p}$ for the natural norms:  
$$
\|v\|_{\beta,p}:=\left[\mathbb{E} \left(  \|v(\omega)\|_{L^{p}\left([0,T]\right)}^{\beta}\right) \right]^{\frac{1}{\beta}} 
\mbox{and } \; \; \; 
\|v\|_{\infty,p}:= \esssup\limits_{\omega\in \Omega}\|v(\omega)\|_{L^{p}\left([0,T]\right)}.
$$
The case $\beta=p=2$ is of particular interest   since $(L_{\FF}^{2,2})^{n}$  is a Hilbert space endowed with the scalar product
$$\langle v_1, v_2\rangle_{L^{2}}:= \EE\left( \int_{0}^{T} v_{1}(t)^{\top} v_{2}(t) \dd t \right).$$
We set $(\M_{c}^{2})^{n}$  for the set consisting of  $\mathbb{F}$-adapted, $\RR^{n}$-valued square integrable martingales $x(\cdot)$ satisfying that $x(0)=0$.  Recall that in the brownian filtration $\mathbb{F}$, every martingale admits a version having $\PP$-almost surely (a.s.)  continuous trajectories (see \cite[Theorem 3.5, Chapter V]{Revuz-Yor}). In particular, the elements in $(\M_{c}^{2})^{n}$ can be identified with  $\FF$-progressively measurable processes. Let us also recall that for every $x\in (\M_{c}^{2})^{n}$, the martingale representation theorem (see e.g. \cite[Chapter 2, Theorem 6.6]{ikedawata}) provides the existence of a unique $x_{2} \in (L^{2,2}_{\FF})^{n\times d}$ such that
\be\label{mamnennennrnnrnrss} x(t)= \int_{0}^{\cdot} x_{2}(s) \dd W(s) \hspace{0.5cm} \forall \; t \in [0,T],\ee
where, denoting $ x_{2}^{ij}:= (x_{2}^{j})^{i}$,
$$\left(\int_{0}^{\cdot} x_{2}(s) \dd W(s)\right)^{i}:= \sum_{j=1}^{d}\int_{0}^{\cdot} x_{2}^{ij}(s) \dd W^{j}(s) \hspace{0.5cm} \mbox{for all $i=1, \hdots, n$.}$$
Note that relation \eqref{mamnennennrnnrnrss}, Doob's inequality and the Itô-isometry for the stochastic integral imply that,   endowed with the scalar product 
$$ \langle x, y \rangle_{\M^{2}_{c}} := \EE\left( x(T)^{\top} y(T) \right), $$
 $(\M_{c}^{2})^{n}$  is a Hilbert space  which is   a   closed subspace of $(L_{\FF}^{2,\infty})^{n}$. We now consider a larger Hilbert space, called Itô space, which is   fundamental in the rest of the article. In order to provide a rigorous definition  
let us consider the application $I: \RR^{n} \times (L^{2,2}_{\FF})^{n} \times (L^{2,2}_{\FF})^{n\times d} \to (L^{2,\infty}_{\FF})^{n}$ defined as 
\be\label{amdmmrnnrnssasa}I(x_0,x_1,x_2)(\cdot):= x_0 + \int_{0}^{\cdot} x_{1}(s) \dd s + \int_{0}^{\cdot} x_{2}(s) \dd W(s).\ee
We have 
\begin{lemma}\label{qnernqnnaaaaAAAa} The application $I$ is well defined, injective and $\exists \; c>0$ such that 
\be\label{amdmmmwmwmw}
\| I(x_0,x_1,x_2)\|_{2,\infty} \leq c\left( |x_0|+ \|x_{1}\|_{2,2} +\sum_{j=1}^{d} \|x_{2}^{j}\|_{2,2}\right).
\ee
\end{lemma}
\begin{proof} There exists a constant $c>0$ such that  for all $t \in [0,T]$,\small
$$ |x(t)|^{2}= \left|x_0 + \int_{0}^{t} x_{1}(s) \dd s + \int_{0}^{t} x_{2}(s) \dd W(s)\right|^{2}\leq c\left( |x_0|^2+ \left| \int_{0}^{t} x_{1}(s) \dd s \right|^2 +\left| \int_{0}^{t} x_{2}(s) \dd W(s) \right|^2\right). $$\normalsize
By  Jensen inequality applied to the first integral we get the existence of $c'>0$ such that 
$$ \sup_{t\in [0,T]} |x(t)|^{2}\leq  c'\left( |x_0|^2+ \int_{0}^{T} \left| x_{1}(s) \right|^2 \dd s +\sup_{t \in [0,T]} \left| \int_{0}^{t} x_{2}(s) \dd W(s) \right|^2\right).$$
Taking the expected value, Doob's inequality and the Itô-isometry property for the stochastic integrals yields to \eqref{amdmmmwmwmw}. Finally, since the only continuous martingales with  finite-variation   are the constants (see e.g.   in \cite[Proposition 1.2]{Revuz-Yor}), we see that $I$ is injective.
\end{proof}\smallskip

We consider the space $\I^{n}$ of  $\RR^{n}$-valued Itô processes  defined by
$$ \I^{n}:= I\left( \RR^{n} \times (L^{2,2}_{\FF})^{n} \times (L^{2,2}_{\FF})^{n\times d}\right).$$
By Lemma \ref{qnernqnnaaaaAAAa}   have that $\I^{n}$ is a linear space which can be identified with $\RR^{n} \times (L^{2,2}_{\FF})^{n} \times (L^{2,2}_{\FF})^{n\times d}$. For $x \in \I^{n}$ we set $(x_0,x_1,x_{2})= I^{-1}(x)$ and we define the scalar product
\be\ba{ll}
\langle x,y \rangle_{\I}&:= x_{0}^{\top}y_{0} + \langle  x_{1}, y_{1} \rangle_{L^{2}} + \sum_{j=1}^{d}\langle  x_{2}^{j}, y_{2}^{j} \rangle_{L^{2}} \hspace{0.5cm} \forall \; x, y \in \ito, \\[4pt]
\; & =   x_0^{\top} y_0 + \langle  x_{1}, y_{1} \rangle_{L^{2}} + \sum_{j=1}^{d}\left\langle \int_{0}^{\cdot}x_{2}^{j}(s)\dd W^{j}(s),\int_{0}^{\cdot}y_{2}^{j}(s)\dd W^{j}(s) \right\rangle_{\M_{c}^2}\\[4pt]
\; & =  x_{0}^{\top}y_{0} + \EE\left(\int_{0}^{T}  x_{1}(t)^{\top} y_{1}(t) \dd t \right) +  \EE\left(\int_{0}^{T}  \mbox{{\rm tr}}\left[ x_{2}(t)^{\top} y_{2}(t)\right] \dd t \right). \ea
\ee
and we define the norm $\|x\|_{\I}:=\sqrt{\langle x, x \rangle_{\I}}$.
\begin{lemma}\label{meqwnenenqneqbebqbs} The space $(\I^{n}, \| \cdot \|_{\I})$ is a Hilbert space which is continuously embedded in $(L^{2,\infty}_{\FF})^{n}$.  
\end{lemma}
\begin{proof} The result is a direct consequence of the fact that $\RR^{n}\times (L^{2,2}_{\FF})^{n} \times (L^{2,2}_{\FF})^{n\times d}$ is a Hilbert space and 
Lemma \ref{qnernqnnaaaaAAAa}.
\end{proof}
\begin{remark} {\rm (i) } We can  thus identify $\I^{n}$ with  $\RR^{n} \times (L^{2,2}_{\FF})^{n} \times (L^{2,2}_{\FF})^{n\times d}$ and by  \eqref{mamnennennrnnrnrss} with $\RR^{n}\times (L^{2,2}_{\FF})^{n} \times (\M_{c}^{2})^{n}$.\smallskip\\
{\rm(ii)} We will identify the topological dual $\I^{*}$ with the space $\I$ itself. 
\end{remark}
 \section{Adjoint operators and Backward Stochastic Differential Equations (BSDEs)}\label{BSDE}
 We start with two basic well-known results. However, since the proofs are short, we provide the details for the reader's convenience.  
 \begin{lemma}\label{memmnnfbbbbbrbrwww}
 Let $x\in (L^{2,\infty}_{\FF})^{n}$ and $r\in (L^{2,2}_{\FF})^{n}$. Then, for every $j=1, \hdots, d$, 
 $$M^{j}(\cdot):= \int_{0}^{\cdot} x(s)^{\top} r(s) \dd W^{j}(s) \hspace{0.3cm} \mbox{ is a martingale.}$$
 \end{lemma}
 \begin{proof}
 Since  $x\in (L^{2,\infty}_{\FF})^{n}$ and $r\in (L^{2,2}_{\FF})^{n}$ we have that the stochastic integral $M^{j}$ is well-defined and is a local-martingale.  By the  Burkholder-Davis-Gundy inequality  (see e.g \cite{KaraShreve91}) we have the existence of a constant $K>0$ such that 
 $$ \EE\left(\sup_{t\in [0,T]} |M^{j}(t)| \right)\leq K \EE\left[ \left(\int_{0}^{T}  |x(s)^{\top} r(s)|^{2} \dd t\right)^{\half}\right] \leq K \| x\|_{2,\infty} \| r\|_{2,2},$$
 where the last inequality follows from the Cauchy-Schwarz inequality. Therefore, by \cite[Theorem 51]{Protter-book}, we have that $M^{j}(\cdot)$ is a martingale with null expectation. \end{proof}\smallskip
 
 Using the above result, the following one is a straightforward consequence of  Itô's Lemma and Lemma \ref{meqwnenenqneqbebqbs}. 
 \begin{lemma}\label{usoito} Let $x$, $y \in \I^{n}$.  Then
$$ \mathbb{E}\left(  x(T)^{\top}  y(T)  \right)=x_{0}^{\top}y_{0}+ \mathbb{E}\left(\int_{0}^{T}\left[ x(t)^{\top}  y_{1}(t) +  y(t)^{\top} x_{1}(t)   + \sum_{j=1}^{d}    (x_{2}^{j}(t))^{\top} y_{2}^{j}(t)  \right] \dd t \right).$$ 

\end{lemma}  
Given a sigma-algebra $\G \subseteq \F$ we write $L^{p}_{\G}:= L^{p}(\Omega, \G, \PP)$. The following Proposition will be useful.
\begin{proposition}\label{msmamsmasaaaaasasasasa} Let $g \in (L^{2}_{\F_{T}})^{n}$    and $a \in (L^{2,2}_{\FF})^{n}$.
Then, for every $z\in \I^{n}$ we have that 
\begin{eqnarray}\label{admamamsma}\ba{rcl}  \EE\left(  g^{\top} z(T) \right) &=&    \left\langle \EE\left(g| \F_{(\cdot)}\right)+ \int_{0}^{\cdot}  \EE(g| \F_{t})  \dd t, z \right\rangle_{\I} , \\[4pt]
		   \EE\left( \int_{0}^{T}  a(t)^{\top} z(t)  \dd t\right) &=& \left\langle \EE(\int_{0}^{T} a(t) \dd t| \F_{(\cdot)})+  \int_{0}^{\cdot}  \EE(\int_{t}^{T} a(s) \dd s| \F_{t})  \dd t, z \right\rangle_{\I}.\ea
\end{eqnarray}
In particular, 
\be\label{mmwnrnwrnnrnrrr}\EE\left(    g^{\top} z(T) +  \int_{0}^{T}    a(t)^{\top} z(t)  \dd t\right) = \left\langle p(0)+ \int_{0}^{\cdot} p(t) \dd t + \int_{0}^{\cdot} q(t) \dd W(t), z\right\rangle_{\I},\ee
where $(p,q) \in \I^{n} \times  (L^{2,2}_{\FF})^{n\times d}$ is the unique solution of the BSDE
$$ \ba{rcl} \dd p&=& -a(t) \dd t + q(t) \dd W(t), \\[2pt]
			p(T)&=& g.\ea$$
\end{proposition}
\begin{proof} Let us first prove  \eqref{admamamsma}.  Let us denote by $r_g$ for  the unique element in $ (L^{2,2}_{\FF})^{n\times d}$ (see \eqref{mamnennennrnnrnrss}) such that 
\be\label{qemqnenqennnenwqqqqq}   \EE\left(g| \F_{(\cdot)}\right)= \EE(g) + \int_{0}^{\cdot} r_{g}(t) \dd W(t).\ee
Lemma \ref{usoito} implies that
$$\ba{rcl} \EE\left( g^{\top} z(T)  \right) &=& \EE\left(  \EE(g| \F_{T})^{\top} z(T)  \right),\\[4pt]
								\; &=&     \EE\left(   \EE(g| \F_{0})^{\top} z_{0}  + \int_{0}^{T} \left[  \EE(g| \F_{t})^{\top} z_1 + \sum_{j=1}^{d}(r_{g}^{j})^{\top}   z_{2}^{j} \right]  \dd t \right),\\[4pt]
								\; &=&    \left\langle \EE(g)+ \int_{0}^{\cdot}  \EE(g| \F_{t})  \dd t+  \int_{0}^{\cdot}  r_g(t) \dd W(t), z \right \rangle_{\I} , \\[4pt]
 \ea$$
which, together with \eqref{qemqnenqennnenwqqqqq}, yields to the first identity in \eqref{admamamsma}. On the other hand,   setting $y(\cdot):= \int_{0}^{\cdot} a(t) \dd t$,
$$ \EE\left( \int_{0}^{T} a(t)^{\top}  z(t)   \dd t\right)  =  \EE\left( \int_{0}^{T}  z(t)^{\top} \dd y(t)   \right) \; \;=    \EE\left(   y(T)^{\top} z(T)  - \int_{0}^{T}   y(t)^{\top} z_{1}(t)  \dd t  \right),$$
 and the second identity in   \eqref{admamamsma} follows from the first one. To establish \eqref{mmwnrnwrnnrnrrr}, let $q\in (L^{2,2}_{\FF})^{n\times d}$ be such that 
 $$ \EE\left( g + \int_{0}^{T} a(t) \dd t \big| \F_{(\cdot)}\right)= \EE\left(  g + \int_{0}^{T} a(t) \dd t \right)+ \int_{0}^{\cdot} q(t) \dd W(t), $$
and define 
$$ p(t):= \EE\left( g + \int_{t}^{T}a(s) \dd s \big| \F_{t} \right).$$
Then
$$p(t)= \EE\left( g + \int_{0}^{T}a(s) \dd t \big| \F_{t} \right)- \int_{0}^{t}a(s) \dd s=p(0)- \int_{0}^{t}a(s) \dd s+  \int_{0}^{t} q(s) \dd W(s),$$
from which the result follows.
\end{proof} \smallskip

 For $g \in (L_{\FF}^{\infty,\infty})^{n\times n}$ and $h=(h^{j})_{j=1}^{d}$ with $h^{j}\in  (L_{\FF}^{\infty,\infty})^{n\times n}$, let us define the operators $A_{g}$, $B_{h}: \I^{n} \to \I^{n}$ as
\be\label{damdamdmaaaa} A_{g} z  :=  \int_{0}^{\cdot} g(s) z(s) \dd s, \hspace{0.8cm}
 B_{h} z  := \sum_{j=1}^{d} \int_{0}^{\cdot} h^{j}(s)  z(s) \dd W^{j}(s). \ee
  Proposition \ref{msmamsmasaaaaasasasasa} has the following   consequence:
\begin{corollary}\label{mqqmqqsaospapsa} The following assertions hold: \smallskip\\
{\rm(i)} The operator $A_{g}$  is continuous and its adjoint $A_{g}^{*}: \I^{n} \to \I^{n}$ is given by
\be\label{mmmmssssaasasasaasasassa} A_{g}^{*} r(\cdot) =  \EE\left(\int_{0}^{T}  g(t)^{\top} r_{1}(t) \dd t \big | \F_{(\cdot)} \right) +\int_{0}^{\cdot}  \EE\left(\int_{t}^{T}  g(s)^{\top} r_{1}(s) \dd s \big | \F_{t} \right)\dd t, \hspace{0.5cm} \forall \; r\in \I^{n}.\ee
Moreover,  $$ A_{g}^{*} r(\cdot)= p_{g,r}(0)+\int_{0}^{\cdot} p_{g,r}(t) \dd t +  \int_{0}^{\cdot}q_{g,r}(t) \dd W(t) \hspace{0.5cm} \forall \; r \in \I^{n},$$
where $(p_{g,r}, q_{g,r})\in \I^{n}\times (L^{2,2}_{\FF})^{n\times d}$ is the unique solution of the following BSDE
$$\ba{rcl} \dd p(t)&=&  -g(t)^{\top} r_1(t) \dd t + q(t) \dd W(t),   \\[4pt]
		       p(T) &=& 0.
\ea$$
{\rm(ii)} The operator $B_{h}$ is continuous and its adjoint $B_{h}^{*}: \I^{n} \to \I^{n}$ is given by \small
\be\label{mmmmssssaasasasasa} B_{h}^{*} r(\cdot) = \sum_{j=1}^{d}\EE\left(\int_{0}^{T}  h^{j}(t)^{\top} r_{2}^{j}(t) \dd t \big | \F_{(\cdot)} \right) + \sum_{j=1}^{d} \int_{0}^{\cdot}  \EE\left(\int_{t}^{T}  h^{j}(s)^{\top} r_{2}^{j}(s) \dd s \big | \F_{t} \right)\dd t, \hspace{0.5cm} \forall \; r\in \I^{n}.  \ee \normalsize
Moreover, 
$$ B_{h}^{*} r(\cdot)= p_{h,r}(0)+\int_{0}^{\cdot} p_{h,r}(t) \dd t +  \int_{0}^{\cdot}q_{h,r}(t) \dd W(t)  \hspace{0.5cm} \forall \; r \in \I^{n},$$
where $(p_{h,r}, q_{h.r})\in  \I^{n} \times (L^{2,2}_{\FF})^{n\times d}$ is the unique solution of the following BSDE
$$\ba{rcl} \dd p(t)&=&  -\sum_{j=1}^{d}h^{j}(t)^{\top} r_2^{j}(t) \dd t + q(t) \dd W(t),   \\[4pt]
		       p(T) &=& 0.
\ea$$
Consequently, the adjoint of  $A_{g} + B_{h}$ is given by
$$(A_{g} + B_{h})^{*}r(\cdot) = p_{r}(0)+\int_{0}^{\cdot} p_{r}(t) \dd t +  \int_{0}^{\cdot}q_{r}(t) \dd W(t)  \hspace{0.5cm} \forall \; r \in \I^{n},$$
where $(p_{r}, q_{r})\in \I^{n} \times (L^{2,2}_{\FF})^{n\times d}$ is the unique solution of the following BSDE
$$\ba{rcl} \dd p(t)&=&  - \left[g(t)^{\top} r_1(t)+\sum_{j=1}^{d}h^{j}(t)^{\top} r_2^{j}(t)\right] \dd t + q(t) \dd W(t),   \\[4pt]
		       p(T) &=& 0.
\ea$$
\end{corollary}
\begin{proof}  For all $z\in \I^{n}$, we have that 
$$\ba{rcl} \| A_g z\|_{\I}& =&  \left[\EE\left(\int_{0}^{T}|g(t)z(t)|^{2} \dd t  \right) \right]^{\half} \leq n\| g \|_{\infty,\infty} \| z\|_{2,2}\leq n\sqrt{T}\| g \|_{\infty,\infty} \| z\|_{2,\infty},\\[4pt] 
		 \| B_h z\|_{\I}& =&  \sum_{j=1}^{d}  \left[\EE\left(\int_{0}^{T}|h^{j}(t)z(t)|^{2} \dd t  \right) \right]^{\half}  \leq nd\| g \|_{\infty,\infty} \| z\|_{2,2}\leq nd\sqrt{T}\| g \|_{\infty,\infty} \| z\|_{2,\infty}. \ea
$$
Therefore, Lemma \ref{meqwnenenqneqbebqbs}  implies that the linear operators are indeed continuous. We also have that, by Lemma \ref{msmamsmasaaaaasasasasa}:
$$\ba{rcl} \langle r, A_{g}z\rangle_{I} &=&  \EE\left( \int_{0}^{T} r_1(t)^{\top} g(t) z(t)\dd t\right), \\[4pt]
					\;  &=& \left\langle \EE( \int_{0}^{T}g(t)^{\top}  r_1(t) \dd t | \F_{(\cdot)})+\int_{0}^{\cdot}  \EE( \int_{t}^{T}g(s)^{\top}  r_1(s) \dd s|\F_{t}) \dd t , z \right\rangle_{\I},
\ea$$
which, by Lemma \ref{msmamsmasaaaaasasasasa}, implies the expression for  $A_{g}^{*}$ in {\rm(i)}. The corresponding identity for $B_{g}^{*}$ in {\rm(ii)} is obtained by an analogous argument, while assertion {\rm(iii)} is a direct consequence of {\rm(i)}-{\rm(ii)}.
\end{proof}\smallskip\\
\section{Optimal control problem and  Lagrange multipliers}\label{mwmnnsndnadawwwapapapapas}
Let us introduce some notations and assumptions. For a differentiable function $(a,b) \in \RR^{n_{1}}\times \RR^{n_{2}}  \to \psi(a,b) \in \RR^{n_{3}}$ we denote by $\psi_{a}(a,b)\in \RR^{n_{3} \times n_{1}}$ and $\psi_{b}(a,b)\in \RR^{n_{3} \times n_{2}}$ the corresponding Jacobian matrices. Let $f: \Omega \times [0,T] \times \RR^{n}\times \RR^{m} \to \RR^{n}$  and $ \sigma: \Omega \times [0,T] \times \RR^{n}\times \RR^{m} \to \RR^{n\times d}$. In what follows we use the notation $\ds f=(f^{i})_{(1\leq i \leq n)}$ and  $\sigma=(\sigma^{ij})_{1\leq i \leq n, \; 1\leq j \leq d }$, where each $f^{i}$ and $\sigma^{ij}$ is real valued. The columns of $\sigma$ are written $\sigma^{j}$ for $j=1, \hdots, d$.
%
We suppose that: \medskip\\
\textbf{(H1)} The maps $\psi=f^{j}, \sigma^{ij}$ satisfy: \smallskip
\\
{\rm(i) } $\psi$ is $\F_{T} \otimes \mathcal{B}([0,T]\times \RR^{n}\times \RR^{m})$-measurable. \smallskip
\\
{\rm(ii) } For a.a. $(\omega, t) \in \Omega \times  [0,T]$ the mapping $(x,u)\to \psi(\omega, t,x,u)$ is $C^{1}$, the application  $(\omega, t) \in \Omega \times [0,T] \to \psi(\omega,t, \cdot, \cdot) \in C^{1}(\RR^{n}\times \RR^{m})$ is progressively measurable and there exists $c_1 >0$ such that  almost surely in   $(\omega, t)$ 
\be\label{acotamientoderivadas}\left\{\ba{c} |\psi(\omega,t,x,u)|\leq  c_{1} \left( 1+ |x|+|u|\right), \\[4pt]
																			|\psi_{x}(\omega, t,x,u)|+ |\psi_{u}(\omega,t,x,u)| \leq c_{1}, \\[4pt] 	
\ea\right.\ee
\begin{remark} Note that under {\bf(H1)} for every $(x,u) \in \ito\times \control$ we  have that $(\omega, t) \to \psi(\omega, t, x(\omega,t), u(\omega,t))$ is progressively measurable, and so   $\int_{0}^{\cdot} f(\omega, t, x(\omega,t), u(\omega,t)) \dd t$  and  $\int_{0}^{\cdot} \sigma(\omega, t, x(\omega,t), u(\omega,t)) \dd W(t)$ are two a.s. continuous progressively measurable processes. The latter is also a  square integrable continuous martingale.
\end{remark}
Let us consider the application $G: \I^{n} \times (L^{2,2}_{\FF})^{m} \to \I^{n}$ defined by
\be\label{meqmemnnrnnrnr}G(x,u) :=   \int_{0}^{\cdot} f(s,x(s),u(s)) \dd s +  \int_{0}^{\cdot} \sigma(s,x(s), u(s)) \dd W(s)-x(\cdot).\ee

 \begin{lemma}\label{wrnwnrnwnrnqbqbqbq}  Under {\rm \textbf{(H1)}} the mapping $G$ is Lipschitz continuous and Gâteaux differentiable.  Its Gâteaux derivative $DG(x,u): \I^{n} \times (L^{2,2}_{\FF})^{m} \mapsto   \I^{n} $ is given by 
\be\label{mmsnbabbasvvwwwaaass}\ba{rcl} DG(x,u)(z,v)(\cdot) &=&  \int_{0}^{\cdot}\left[   f_{x}(t,x(t),u(t)) z(t) +   f_{u}(t,x(t),u(t)) v(t)\right] \dd t\\[4pt]
 			\;     & \; & +\int_{0}^{\cdot}\left[   \sigma_{x}(t,x(t),u(t)) z(t) +   \sigma_{u}(t,x(t),u(t)) v(t)\right] \dd W(t)-z(\cdot),	\ea\ee
for all $(z,v) \in \I^{n} \times  (L^{2,2}_{\FF})^{m}$. Moreover, for every $u$, $v\in (L^{2,2}_{\FF})^{m}$ and every $x\in\I^n$  we have that  $D G(x,u)(\cdot, v): \ito \mapsto \ito$ is bijective.
 \end{lemma} 
 \begin{proof}  Given $z\in \ito$, $v\in \control$ and $\tau>0$, by a first order Taylor expansion of $f$ and $\sigma$ we obtain \footnotesize
 \be\label{memqemnndndasas}\ba{rcl} G(x+\tau z, u+\tau v)- G(x,u) &=&  \tau \int_{0}^{\cdot}\left[   f_{x}(t,x(t),u(t)) z(t) +   f_{u}(t,x(t),u(t)) v(t)+r_{1}(t,\tau)\right] \dd t \\[4pt]
 									\; & \; & + \tau \int_{0}^{\cdot}\left[   \sigma_{x}(t,x(t),u(t)) z(t) +   \sigma_{u}(t,x(t),u(t)) v(t) + r_{2}(t,\tau)\right] \dd W(t)\\[4pt]
									\; & \; & -\tau z(\cdot),\ea\ee\normalsize
where
$$\ba{rcl} r_{1}(\omega,t,\tau) &:=&  \int_{0}^{1}\left[ Df(t,x(t)+ \theta \tau z(t),u(t)+ \theta \tau v(t))- D f(t,x(t),u(t)) \right](z,v)\dd \theta,\\[4pt]
r_{2}(\omega,t,\tau) &:=&  \int_{0}^{1}\left[ D\sigma(t,x(t)+ \theta \tau z(t),u(t)+ \theta \tau v(t))- D \sigma(t,x(t),u(t)) \right](z,v)\dd \theta.
\ea
$$
By {\bf (H1)}{\rm(ii)}, we have that 
\be\label{memqmernansnansddda}   |r_{1}(\omega,t,\tau)|^{2}+ |r_{2}(\omega,t,\tau)|^{2} \leq   c'\left( |z(\omega,t)|^{2}+ |v(\omega,t)|^{2}\right) \hspace{0.2cm} \mbox{for a.a. $(\omega,t)\in \Omega \times [0,T]$.}   \ee
Since  the left hand side of \eqref{memqmernansnansddda} converges a.s. to $0$  as $\tau \downarrow 0$, we deduce with Lemma \ref{meqwnenenqneqbebqbs} and the dominated convergence theorem that
$$ \EE\left(  \int_{0}^{T} |r_{1}(t,\tau)|^{2} \dd t \right)+\EE\left(\int_{0}^{T}  |r_{2}(t,\tau)|^{2}\dd t \right) \to 0 \hspace{0.3cm} \mbox{as $\tau\downarrow 0$},$$
and thus \eqref{mmsnbabbasvvwwwaaass} follows from dividing by $\tau$ in \eqref{memqemnndndasas}, taking the limit $\tau \downarrow 0$ and the definition of convergence in $\ito$. Now, fix $v\in \control$ and $\xi \in \ito$.  Let us prove that there exists $z\in \ito$ such that $DG(x,u)(z,v)=\xi$. By definition, this is equivalent to solving the SDE
$$\ba{rcl} \dd z &=& \left[  f_{x}(t,x(t),u(t)) z(t) +   f_{u}(t,x(t),u(t)) v(t)- \xi_{1} \right] \dd t \\[4pt]
			\; & \; &  + \left[   \sigma_{x}(t,x(t),u(t)) z(t) +   \sigma_{u}(t,x(t),u(t)) v(t)-\xi_2\right]\dd W(t)\\[4pt]
	               z(0)&=& - \xi_{0}.\ea$$
Since $(\xi_1, \xi_{2})\in (L^{2,2}_{\FF})^{n}\times(L^{2,2}_{\FF})^{n\times d}$, under {\bf (H1)} classical results for solvability of linear SDEs (see e.g.  \cite[Theorem 2.1]{Bismut76a}) imply that the above  equation has a unique solution.
 \end{proof}

 \begin{remark}\label{admamdmamnnnssssaaaa} Note that under our assumptions $G$ is Lipschitz.. Therefore, by  classical results {\rm(}see e.g. {\rm \cite[Proposition 2.49]{BonSha})} we have  that  $G$ is Hadamard differentiable, i.e. 
 $$ \lim_{\tau \to 0, \\ (z',v') \to (z,v)} \frac{G(x+\tau z',u+\tau v')(\cdot)-G(x,u)(\cdot)}{\tau}= DG(x,u)(z,v)(\cdot)   \hspace{0.5cm} \mbox{in $\I^{n}$}.$$
In general, it is not clear that $G$ is $C^{1}$. However, if $f$ and $\sigma$ are affine functions of the pair $(x,u)$, it can be easily checked that  $(x,u)\in \ito \times \control  \to DG(x,u) \in L(\ito, \ito)$ is continuous {\rm(}$L(\ito, \ito)$ is the space of bounded linear applications from $\ito$ to $\ito${\rm)}, which implies that $G$ is  continuously differentiable.
 \end{remark}
 
Now, let 
$$ \ell: \Omega \times [0,T] \times \RR^{n} \times \RR^{m} \to \RR,   \hspace{0.4cm} \Phi: \Omega \times \RR^{n} \to \RR \hspace{0.4cm}   \Phi_{E}:  \Omega \times \RR^{n} \to \RR^{n_{E}},   \hspace{0.4cm}  \Phi_{I}:  \Omega \times \RR^{n} \to \RR^{n_{I}}. $$ 
\textbf{(H2)}   We suppose that\\
{\rm(i) } The maps $\ell$ and $\psi= \Phi$, $\Phi_{E}^{i}$, $\Phi_{I}^{j}$  ($1\leq i \leq n_{E}$ and $1\leq j \leq n_{I}$)  are respectively  $\F_{T} \otimes  \mathcal{B}([0,T]\times \RR^{n}\times \RR^{m})$  and $\F_{T} \otimes \mathcal{B}(\RR^{n})$ measurable.\smallskip \\
{\rm(ii) } For a.a. $(\omega,t)$ the maps $(x,u)\to \ell(\omega, t,x,u)$ and     $x\to \psi(\omega, x)$ are $C^1$. The application  $(\omega, t) \in \Omega \times [0,T] \to \ell(\omega,t, \cdot, \cdot) \in C^{1}(\RR^{n}\times \RR^{m})$ is progressively measurable. In addition, there exists $c_{2}>0$   such that  almost surely in   $(\omega, t)$ we have that 
\be\label{crecimientoderivadascosto}\left\{\ba{c}|\ell (\omega , t,x,u)|\leq  c_{2} \left( 1+ |x|+|u|\right)^{2}, \\[4pt]
 |\ell_{x}(\omega,t,x,u)|+ |\ell_{u}(\omega, t,x,u)|\leq  c_{2} \left( 1+ |x|+|u|\right), \\[4pt]
 \ |\psi (\omega, x)|\leq  c_{2} \left( 1+ |x|\right)^{2}, \; |\psi_{x}(\omega,x)|\leq  c_{2} \left( 1+ |x|\right). 
 \ea\right.\ee
 We define $F: \ito \times \control \to \RR$, $G_{E}: \ito  \to \RR^{n_{E}}$ and  $G_{I}: \ito \to \RR^{n_{I}}$ as
\be\label{mamsmnnrnnrnrs}
\ba{rcl}    F(x,u)&:=&  \EE \left( \int_{0}^{T}\ell(t,x(t),u(t)) \dd t + \Phi(x(T)) \right), \\[4pt]
	       G_{E}^{i}(x)&:=&    \EE \left( \Phi_{E}^{i} (x(T)) \right) \hspace{0.5cm} \forall \; i=1,\hdots, n_{E}, \\[4pt]
	        G_{I}^{j}(x)&:=&    \EE \left( \Phi_{I}^{j} (x(T)) \right) \hspace{0.5cm} \forall \; j=1,\hdots, n_{I}.\ea
\ee
\begin{lemma}\label{qwemqemqemaaaaa} The functions $F$, $G_{E}$ and $G_{I}$ are continuously  differentiable (in the Fréchet sense) and $\forall \; (x,u)$, $(z,v) \in \ito \times \control$ we have that
\be\label{admadmamdada}\ba{rcl}
 DF(x,u)(z,v)	& = & \EE\left(\int_{0}^{T}\left[ \ell_x(t,x(t),u(t)) z(t) + \ell_{u}(t,x(t),u(t))  v(t)\right] \dd t+  \Phi_{x}(x(T)) z(T) \right),\;\\
 DG_{E}^{i}(x,u)(z,v) 	& = & \EE\left(   (\Phi_{E}^{i})_{x}(x(T)) z(T) \right) \hspace{0.5cm} \forall \; 1\leq i \leq n_{E},\\
 DG_{I}^{j}(x,u)(z,v) 	& = & \EE\left(   (\Phi_{I}^{j})_{x}(x(T)) z(T) \right) \hspace{0.5cm} \forall \; 1\leq j \leq n_{I}.
 \ea\ee
 \end{lemma}
 \begin{proof} The proof that  $F$ is Gâteaux differentiable and that its Gâteaux derivative satisfies the first equation in \eqref{admadmamdada} follows the same lines as the proof of Lemma \ref{wrnwnrnwnrnqbqbqbq}. Now,  note that given $(z, v) \in \ito \times \control$, with $\|z\|_{\I}=\|v\|_{2,2}=1$,  for all $(x,u)$, $(x',u') \in \ito \times \control$ \small
 $$\ba{rcl} |  DF(x,u)(z,v)-DF(x',u')(z,v)|& \leq &\| z \|_{2,\infty}\left( \EE\left[ \int_{0}^{T}\left| \ell_{x}(t,x(t),u(t))- \ell_{x}(t,x'(t),u'(t))\right| \dd t \right]^{2} \right)^{\half}\\
 							\; & \; & + \| z \|_{2,\infty}\left( \EE\left[ \left|\Phi_{x}(x(T)) - \Phi_{x}(x'(T)) \right|^{2} \right] \right)^{\half}\\
							\; & \; & +\| v \|_{2,2}\left( \EE\left[ \int_{0}^{T}\left| \ell_{u}(t,x(t),u(t))-\ell_{u}(t,x'(t),u'(t))\right|^{2} \dd t \right]  \right)^{\half}.\ea$$
\normalsize
Therefore, by  Lemma \ref{qnernqnnaaaaAAAa} we get that \small
$$ \sup_{(z,v) \in \ito \times \control  \; ; \; \|z\|_{\I}=\|v\|_{2,2}=1} |  DF(x,u)(z,v)-DF(x',u')(z,v)|^{2}\leq c w(x',u'),$$\normalsize
where \small
 $$\ba{ll}w(x',u'):= & \EE\left( \int_{0}^{T}\left[ \left| \ell_{x}(t,x(t),u(t))- \ell_{x}(t,x'(t),u'(t))\right|^{2} +\left| \ell_{u}(t,x(t),u(t))-\ell_{u}(t,x'(t),u'(t))\right|^{2} \right]\dd t \right.\\[4pt]
 	\; & \left. +  \left|\Phi_{x}(x(T)) - \Phi_{x}(x'(T)) \right|^{2}\right).\ea $$\normalsize
Since $\ell_{x}$, $\ell_u$ and $\Phi_x$ satisfy the linear growth property in \eqref{crecimientoderivadascosto}, we have by dominated convergence that $w(x',u') \to 0$ as $\|x'-x\|_{\I}+\| u'-u\|_{2,2}\to 0$. Thus $DF$ is continuous and therefore $F$ is Fréchet differentiable. The proof of the analogous result for $G_{E}$ and $G_{I}$ follows the same lines.\end{proof}

Let $U \subseteq \RR^{m}$ be a non-empty, closed and convex set and define 
\be\label{qmemqmenandadds}\U:= \left\{ u \in \control \; ; \; u(\omega,t) \in U \; \; \mbox{for a.a. $(\omega,t)\in \Omega \times [0,T]$}\right\}.\ee
We consider the optimal control problem 
$$ \ba{l}\mbox{Min}_{x\in \ito, u\in \control} \ \   F(x,u)  \  \ \mbox{s.t. } \  G(x,u) +x_{0}= 0, \ \   G_{E}(x)=0  \hspace{0.2cm} \mbox{and } \;   G_{I}(x) \leq  0, \; \; u\in  \U.\ea \eqno{(\mathcal{SP})}$$
\begin{remark}\label{qemqmemnsnsnnsnsnfbfbf} Usually the optimal control problem above is stated  only in terms of $u$. Indeed, under our assumptions, for every $u \in \control$ there exists a unique $x[u] \in \I^{n}$ such that $G(x[u], u)+x_0=0$. Therefore, problem $(\mathcal{SP})$ can be equivalently written as 
$$ \ba{l}\mbox{{\rm Min} }_u \ \   F(x[u],u)  \  \ \mbox{{\rm s.t. }} \    G_{E}(x[u])=0 \hspace{0.2cm} \mbox{{\rm and }} \;   G_{I}(x[u]) \leq  0, \; \; u\in  \U.\ea \eqno{(\mathcal{SP'})}$$
We have preferred to consider the minimization problem in terms of the pair $(x,u)$ and thus to maintain explicitly the  constraint $G(x,u)+x_{0}=0$  in order to associate a Lagrange multiplier to it. 
\end{remark}
\begin{definition}{\rm(i)} We say that $(x,u)\in \ito \times \control$ is feasible for $(\mathcal{SP})$ if $G(x,u) + x_0= 0$, $G_{E}(x)=0$,  $G_{I}(x) \leq  0$ and $u \in  \U$.  The set of feasible pairs for problem  $(\mathcal{SP})$  is denoted by $F(\mathcal{SP})$.\smallskip\\
{\rm(ii)} We say that $(\bar{x}, \bar{u}) \in F(\mathcal{SP})$ is a local solution of $(\mathcal{SP})$ iff $\exists \; \eps>0$ such that $F(\bar{x},\bar{u})\leq F(x,u)$ for all $(x,u)\in F(\mathcal{SP})$ satisfying that $\|x-\bar{x}\|_{\I}+ \|u-\bar{u}\|_{2,2}\leq \eps$.
\end{definition}

\subsection{Weak-Pontryagin multipliers and Lagrange multipliers}\label{descripcionpmp}
Given $\alpha \geq 0$ the {\it Hamiltonian} $H[\alpha]: \Omega \times [0,T] \times \RR^{n} \times \RR^{m} \times \RR^{n} \times \RR^{n\times d} \to \RR$ is defined as
\be\label{mamdndnbbbrbbrbr} 
H[\alpha](\omega,t,x, u, p,q):= \alpha \ell(\omega,t, x, u) +   p^{\top} f(\omega,t, x,u)  + \sum_{j=1}^{d}   (q^{j})^{\top} \sigma^{j}(\omega, t,x,u).
\ee

\begin{definition}[weak-Pontryagin multiplier]\label{mameqnennenee} We say that  $0\neq (\bar{\alpha}, \bar{p},\bar{q}, \bar{\lambda}_{E}, \bar{\lambda}_{I}) \in  \RR \times \ito \times (L^{2,2}_{\FF})^{n\times d}  \times \RR^{n_{E}} \times \RR^{n_{I}}$  is a generalized weak-Pontryagin multiplier at $(\bar{x}, \bar{u})$ if 
\be\label{annqenbrbbrbsbsbsbsaaa}
\ba{rcl} 
\dd \bar{p}(t) &=& -H_{x}[\bar{\alpha}](t,\bar{x}(t), \bar{u}(t), \bar{p}(t), \bar{q}(t))^{ \top} \dd t +  \bar{q}(t) \dd W(t),\\[4pt]
       \bar{p}(T)&=&\bar{\alpha}\Phi_{x}(\bar{x}(T))^{\top}   +  (\Phi_{E})_{x}(\bar{x}(T))^{\top}\bar{\lambda}_{E}+   (\Phi_{I})_{x}(\bar{x}(T))^{\top}\bar{\lambda}_{I}   ,\\[4pt]
       0 &\leq &  H_{u}[\bar{\alpha}](\omega,t,\bar{x}(t), \bar{u}(t), \bar{p}(t), \bar{q}(t))(v-\bar{u}(\omega,t)) \hspace{0.3cm} \mbox{{\rm   $\forall v\in U$ a.a.} $(\omega,t)\in \Omega \times [0,T]$,} \\[4pt]
       0  &< & |\bar{\alpha}| + | \bar{\lambda}_{I}| + | \bar{\lambda}_{E}|,\\[4pt]
        0&=&  \bar{\lambda}_{I}^{j} G_{I}^{j}(\bar{x}(T)) \hspace{0.4cm} \forall \; j=1, \hdots, n_{I}, \\[4pt]
        0&\leq & \bar{\lambda}_{I}^{j}  \hspace{0.4cm} \forall \; j =1, \hdots, n_{I} \hspace{0.2cm} \mbox{and } \; 0\leq \bar{\alpha}.
\ea
\ee
If $\bar{\alpha}> 0$ (and therefore can be normalized to $\bar{\alpha}=1$), we say that $0\neq (\bar{p},\bar{q}, \bar{\lambda}_{E}, \bar{\lambda}_{I})$ is a  weak-Pontryagin multiplier at $(\bar{x}, \bar{u})$. The set of weak-Pontryagin multipliers is denoted by $\Lambda_{wP}(\bar{x},\bar{u})$.
\end{definition}  
 It is well known that the following stochastic {\it weak-Pontryagin minimum principle} holds (see e.g.  \cite{Peng90,MouYong07} and \cite[Theorem 3.2, Chapter 3]{YongZhou})
 \begin{theorem}[weak-Pontryagin minimum principle]\label{pmp}
  Assume that {\bf(H1)}-{\bf(H2)} hold and let $(\bar{x},\bar{u}) \in \ito \times  \control$ be a local solution of $(SP)$.  Then, there exists at least one weak-Pontryagin multiplier at $(\bar{x}, \bar{u})$.
 \end{theorem}
\begin{remark} {\rm(i)} In view of Remark \ref{qemqmemnsnsnnsnsnfbfbf},   Pontryagin principles are usually stated for a local solution $\bar{u}$ of $(\mathcal{SP'})$.  However, we easily check that $\bar{u}$ is a local solution of  $(\mathcal{SP'})$ if and only if $(\bar{x}, \bar{u})$ is a local solution of  $(\mathcal{SP})$.\smallskip\\
{\rm(ii)} We called  the result of Theorem \ref{pmp} a weak-Pontryagin minimum principle, since in general more information can be obtained. In fact, even when $U$ is not convex, under a Lipschitz type assumption on the second derivatives of the data,  a second pair of adjoint processes can be introduced in such a manner that the optimal $\bar{u}$ minimizes an associated Hamiltonian in $U$. In the particular case when $U$ is convex, \eqref{annqenbrbbrbsbsbsbsaaa} is an  easy consequence of this result {\rm(}see e.g. {\rm\cite{Peng90}} and {\rm \cite[Chapter 3]{YongZhou}}{\rm)}.\smallskip\\
\end{remark}
The {\it Lagrangian} $\L: \ito \times \control \times \ito \times \RR \times \RR^{n_{E}} \times \RR^{n_{I}} \to \RR$ associated to problem $(\mathcal{SP})$ is defined by
\be\label{amsmmnrnrnrnrds}
\L(x,u,\alpha, \lambda_{\I}, \lambda_{E}, \lambda_{I}):=  \alpha F(x,u) + \langle \lambda_{\I}, G(x,u)+x_0\rangle_{\I} +   \lambda_{E}^{\top}  G_{E}(x)    +  \lambda_{I}^{\top} G_{I}(x),
\ee
where $G: \ito \times \control \mapsto \ito$ is defined in \eqref{meqmemnnrnnrnr} and  $F$, $G_{E}$ and $G_{I}$ are defined in \eqref{mamsmnnrnnrnrs}. 
\begin{definition}\label{mmmsskkakakaaaaaa} We say that $0 \neq (\bar{\alpha}, \bar{\lambda}_{\I}, \bar{\lambda}_{E}, \bar{\lambda}_{I})$ is a generalized Lagrange multiplier at $(\bar{x},\bar{u})$ if
\be\label{mmdammiirnr}\ba{rcl} 
0&=&D_{x}  \L(\bar{x},\bar{u},\bar{\alpha},\bar{\lambda}_{\I} ,\bar{\lambda}_{E}, \bar{\lambda}_{I}), \\[4pt]
0&\leq& D_{u}  \L(\bar{x},\bar{u},\bar{\alpha},\bar{\lambda}_{\I} ,\bar{\lambda}_{E}, \bar{\lambda}_{I})(v-\bar{u}) \hspace{0.5cm} \forall v \in  \U, \\[4pt]
0  &< & |\bar{\alpha}| + |\bar{\lambda}_{I}| + |\bar{\lambda}_{E}|,\\[4pt]
        0&=&   \bar{\lambda}_{I}^{j} G_{I}^{j}( \bar{x}(T)) \hspace{0.4cm} \forall \; j =1, \hdots, n_{I}, \\[4pt]
        0&\leq & \bar{\lambda}_{I}^{j}  \hspace{0.4cm} \forall \; j=1, \hdots, n_{I} \hspace{0.2cm} \mbox{and } \; 0\leq \bar{\alpha}.
\ea\ee 
If $\bar{\alpha} > 0$ (and therefore can be normalized to $\bar{\alpha}=1$) we will say that $0\neq (\bar{\lambda}_{\I}, \bar{\lambda}_{E}, \bar{\lambda}_{I})$ is a Lagrange multiplier at $(\bar{x},\bar{u})$ and we will eliminate the $\bar{\alpha}$ from the arguments of $\L$. The set of Lagrange multipliers is denoted by $\Lambda_{L}(\bar{x},\bar{u})$.
\end{definition}
\begin{remark}  If no final constraints are present, we will eliminate $(\lambda_{E}, \lambda_{I})$ from the arguments of $\L$
\end{remark}
Using the theoretical framework introduced in the previous sections we can prove the following 
\begin{theorem} 
\label{Teoidentificacion}
Let $(\bar{x},\bar{u}) \in  F(\mathcal{SP})$.
If $( \bar{\alpha}, \bar{p}, \bar{q}, \bar{\lambda}_{E}, \bar{\lambda}_{I})$ is a generalized weak-Pontryagin multiplier at $(\bar{x},\bar{u})$  then 
 $(\bar{\alpha}, \bar{\lambda}_{\I}, \bar{\lambda}_{E}, \bar{\lambda}_{I})$ is a generalized Lagrange multiplier at $(\bar{x},\bar{u})$, where
 \be\label{mmmdhhhadhhahahbrbr}
\bar{\lambda}_{\I}(\cdot):= \bar{p}(0)+ \int_{0}^{\cdot} \bar{p}(s) \dd s +  \int_{0}^{\cdot} \bar{q}(s) \dd W(s). \ee 
Conversely, if  $(\bar{\alpha}, \bar{\lambda}_{\I}, \bar{\lambda}_{E}, \bar{\lambda}_{I})$ is a generalized Lagrange multiplier at $(\bar{x},\bar{u})$   then $(\bar{\lambda}_{\I})_{0}= (\bar{\lambda}_{\I})_{1}(0)$ and $( \bar{\alpha},(\bar{\lambda}_{\I})_{1}, (\bar{\lambda}_{\I})_{2}, \bar{\lambda}_{E}, \bar{\lambda}_{I})$ is a generalized weak-Pontryagin multiplier  at $(\bar{x},\bar{u})$.

\end{theorem}
\begin{remark} If $\bar{\alpha}=1$ we can replace in the statement of the theorem ``generalized weak-Pontryagin multiplier'' by ``weak-Pontryagin multiplier''  and ``generalized Lagrange multiplier'' by ``Lagrange multiplier''.
\end{remark}
\begin{proof} For notational convenience we set 
$\ell_{x}(t):=\ell_{x}(t,\bar{x}(t),\bar{u}(t))$, $\sigma_x(t):=\sigma_{x}(t,\bar{x}(t),\bar{u}(t))$ with analogous definitions for $f_u(t)$ and $\sigma_u(t)$. Let $( \bar{\alpha}, \bar{p}, \bar{q}, \bar{\lambda}_{E}, \bar{\lambda}_{I})$ be a generalized weak-Pontryagin multiplier at $(\bar{x},\bar{u})$. In order to prove that $(\bar{\alpha}, \bar{\lambda}_{\I}, \bar{\lambda}_{E}, \bar{\lambda}_{I})$, with $\bar{\lambda}_{\I}$ given by \eqref{mmmdhhhadhhahahbrbr}, is a generalized Lagrange multiplier at $(\bar{x},\bar{u})$ it suffices to show that  the first two relations in \eqref{mmdammiirnr} hold true.  For the first one, for every $z\in \ito$, Lemma \ref{wrnwnrnwnrnqbqbqbq} and Lemma \ref{qwemqemqemaaaaa} imply that \small
\be\label{madnandadnadaaaaaaaaa}\ba{rcl} \bar{\alpha} D_{x}F(x,u)z&=&  \EE\left(\int_{0}^{T} \bar{\alpha} \ell_{x}(t)z(t) \dd t+  \bar{\alpha}\Phi_{x}(\bar{x}(T)) z(T)\right), \\[6pt]
									   		 \langle \bar{\lambda}_{E}, D_{x}G_{E}(\bar{x}) z \rangle&=&  \EE\left(  \bar{\lambda}_{E}^{\top} (\Phi_{E})_{x} (\bar{x}(T))   z(T)  \right),\\[6pt]
											 \langle \bar{\lambda}_{I}, D_{x}G_{I}(\bar{x})z \rangle&=& \EE\left(  \bar{\lambda}_{I}^{\top} (\Phi_{I})_{x} (\bar{x}(T))   z(T)  \right),\\[4pt]
											 \langle \bar{\lambda}_{\I}, D_{x}G(\bar{x})z \rangle_{\I}&=& \EE\left(\int_{0}^{T} \left[   \bar{\lambda}_{1}(t)^{\top}f_{x} (t) + \sum_{j=1}^{d}\bar{\lambda}_{2}^{j}(t)^{\top}  \sigma_{x}^{j}(t)\right] z(t)\dd t \right)- \langle \bar{\lambda}_{\I},z \rangle_{\I}.
						\ea\ee\normalsize
Using Proposition \ref{msmamsmasaaaaasasasasa}, with $a= \bar{\alpha} \ell_{x}(t)$  and $g^{\top}=	\bar{\alpha}\Phi_{x}(\bar{x}(T))+\bar{\lambda}_{E}^{\top} (\Phi_{E})_{x} (\bar{x}(T))+		 \bar{\lambda}_{I}^{\top} (\Phi_{I})_{x} (\bar{x}(T)),$ we get, recalling \eqref{damdamdmaaaa}, \small
$$\ba{rcl} D_{x}  \L(x,u,\alpha,\lambda_{\I} ,\lambda_{E}, \lambda_{I})z&=& \left\langle \hat{p}(0)+\int_{0}^{\cdot} \hat{p}(t) \dd t +\int_{0}^{\cdot} \hat{q}(t) \dd W(t)- \bar{\lambda}_{\I}, z\right \rangle_{\I} +\left\langle \lambda_{\I}, (A_{f_{x}} + B_{\sigma_{x}}) z \right\rangle_{\I},\\[5pt]
					\; &=& \left\langle \hat{p}(0)+\int_{0}^{\cdot} \hat{p}(t) \dd t +\int_{0}^{\cdot} \hat{q}(t) \dd W(t)+ (A_{f_{x}}+B_{\sigma_{x}})^{*} \bar{\lambda}_{\I}- \bar{\lambda}_{\I}, z\right \rangle_{\I},\ea
$$\normalsize
where $(\hat{p}, \hat{q}) \in \ito \times (L^{2,2}_{\FF})^{n \times d}$ is the unique solution of 
$$\ba{rcl} \dd \hat{p}(t)&=&  -\bar{\alpha} \ell_{x}(t)^{\top} \dd t + \hat{q}(t) \dd W(t),   \\[4pt]
		       \hat{p}(T) &=& \bar{\alpha}\Phi_{x}(\bar{x}(T))^{\top} +  (\Phi_{E})_{x}(\bar{x}(T))^{\top}\bar{\lambda}_{E}+   (\Phi_{I})_{x}(\bar{x}(T))^{\top}\bar{\lambda}_{I}.
\ea$$
By Corollary \ref{mqqmqqsaospapsa} we get that  
\be\label{mwmrnnnrnrnnsss} D_{x}  \L(\bar{x},\bar{u},\alpha,\lambda_{\I} ,\lambda_{E}, \lambda_{I})z= \left\langle p(0)+\int_{0}^{\cdot} p(t) \dd t +\int_{0}^{\cdot} q(t) \dd W(t)- \bar{\lambda}_{\I}, z\right \rangle_{\I},\ee
 where $(p, q) \in \ito \times (L^{2,2}_{\FF})^{n \times d}$ is the unique solution of 
\be\label{mmmeeessssaaaaa}\ba{rcl} \dd p(t)&=&  -\left[\bar{\alpha} \ell_{x}(t)^{\top }+f_{x} (t)^{\top}(\bar{\lambda}_{\I})_{1}(t) +   \sigma_{x}(t)^{\top}(\bar{\lambda}_{\I})_{2}(t)\right] \dd t + q(t) \dd W(t),   \\[4pt]
		       p(T) &=&  \bar{\alpha}\Phi_{x}(\bar{x}(T))^{ \top }+  (\Phi_{E})_{x}(\bar{x}(T))^{\top}\bar{\lambda}_{E}+   (\Phi_{I})_{x}(\bar{x}(T))^{\top}\bar{\lambda}_{I}.
\ea\ee
Since $((\bar{\lambda}_{\I})_{1}, (\bar{\lambda}_{\I})_{2})=(\bar{p}, \bar{q})$, by  \eqref{annqenbrbbrbsbsbsbsaaa} we get that $p(T)-\bar{p}(T)=0$ and $\dd [p-\bar{p}](t)= [q(t)-\bar{q}(t)] \dd W(t)$ which yields to   $p=\bar{p}$, $q=\bar{q}$ and in particular $p(0)=\bar{p}(0)$, hence the first relation in  \eqref{mmdammiirnr}  follows from   \eqref{mwmrnnnrnrnnsss}. In order to prove the second relation in  \eqref{mmdammiirnr}  it suffices to note that for all $v\in \U$  \small
\be\label{mmmdmmdmdssaaaa}D_{u}  \L(\bar{x},\bar{u},\bar{\alpha},\bar{\lambda}_{\I} ,\bar{\lambda}_{E}, \bar{\lambda}_{I})(v-\bar{u})= \EE \left( \int_{0}^{T}H_{u}[\bar{\alpha}](\omega,t,\bar{x}(t), \bar{u}, \bar{p}(t), \bar{q}(t))(v(t)-\bar{u}(t)) \dd t\right)\geq 0.\ee \normalsize
Now, let $(\bar{\alpha}, \bar{\lambda}_{\I}, \bar{\lambda}_{E}, \bar{\lambda}_{I})$ be a generalized Lagrange multiplier at $(\bar{x},\bar{u})$. By the first relation in \eqref{mmdammiirnr} and \eqref{mwmrnnnrnrnnsss} we obtain that 
$$ \bar{\lambda}_{\I}=p(0)+\int_{0}^{\cdot} p(t) \dd t +\int_{0}^{\cdot} q(t) \dd W(t),$$
where $(p, q) \in \I^{n} \times (L^{2,2}_{\FF})^{n \times d}$ solves \eqref{mmmeeessssaaaaa}. Therefore, we get $ (\bar{\lambda}_{\I})_{1}=p$ and $ (\bar{\lambda}_{\I})_{2}=q$ and  $ (\bar{\lambda}_{\I})_{0}=p(0)=(\bar{\lambda}_{\I})_{1}(0)$.  Thus  \eqref{mmmeeessssaaaaa} implies that $((\bar{\lambda}_{\I})_{1},(\bar{\lambda}_{\I})_{2})$ satisfies the first and second relations in \eqref{annqenbrbbrbsbsbsbsaaa}. Finally, by  the second relation in  \eqref{mmdammiirnr} and expression \eqref{mmmdmmdmdssaaaa}, we obtain  the third relation in \eqref{annqenbrbbrbsbsbsbsaaa} following the same argument that in the proof of \cite[Theorem 1.5]{CanKara95}.
\end{proof}

As a consequence of the above result we obtain the following sufficient condition, under convexity assumptions. The proof is standard, but since it is very short we provide it for the reader's convenience. 

\begin{corollary}[Sufficient condition for convex problems]\label{memmndnaaaaasssasasasasasasasa}
\label{sufPMP}
Suppose that  $F$ and $G_{I}$ are convex and that $G$ and $G_{E}$ are affine. \smallskip\\
{\rm(i)} Let $(\bar{x},\bar{u}) \in  F(\mathcal{SP})$  and  suppose that  $(\bar{p}, \bar{q}, \bar{\lambda}_{E}, \bar{\lambda}_{I}) \in \ito \times (L^{2,2}_{\FF})^{n\times d} \times \RR^{n_{E}} \times \RR^{n_{I}}$ is a weak-Pontryagin multiplier at $(\bar{x},\bar{u})$.    Then, the pair $(\bar{x}, \bar{u})$ solves $(\mathcal{SP})$.\smallskip\\
{\rm(ii)}  The set of weak-Pontryagin multipliers is independent of the solutions of $(\mathcal{SP})$. More precisely, let $(\bar{x}^{1}, \bar{u}^{1})$, $(\bar{x}^{2}, \bar{u}^{2})  \in   F(\mathcal{SP})$  be two solutions of $(\mathcal{SP})$. Then, $(\bar{p}, \bar{q}, \bar{\lambda}_{E}, \bar{\lambda}_{I})$  is a weak-Pontryagin multiplier at $(\bar{x}^{1},\bar{u}^{1})$ if and only if   it is a weak-Pontryagin multiplier at $(\bar{x}^{2},\bar{u}^{2})$.
\end{corollary}
\begin{proof}  By Theorem  \ref{Teoidentificacion}, $\bar{\lambda}_{\I} \in \I^{n}$ defined by \eqref{mmmdhhhadhhahahbrbr} is a such that $(\bar{\lambda}_{\I}, \bar{\lambda}_{E}, \bar{\lambda}_{I})$ is a Lagrange multiplier. Now, let   $(x,u)$ be feasible for $(\mathcal{SP})$, then  by the convexity of  $\L(\cdot,\cdot,1, \bar{\lambda}_{\I}, \bar{\lambda}_{E}, \bar{\lambda}_{I})$, 
$$\ba{ll}F(x,u) \geq \L(x,u, 1, \bar{\lambda}_{\I}, \bar{\lambda}_{E}, \bar{\lambda}_{I}) \geq&  \L(\bar{x},\bar{u}, 1, \bar{\lambda}_{\I}, \bar{\lambda}_{E}, \bar{\lambda}_{I})+ D_{x}\L(\bar{x},\bar{u}, 1, \bar{\lambda}_{\I}, \bar{\lambda}_{E}, \bar{\lambda}_{I})(x-\bar{x}) \\[4pt]
\; & + D_{u}\L(\bar{x},\bar{u}, 1, \bar{\lambda}_{\I}, \bar{\lambda}_{E}, \bar{\lambda}_{I})(u-\bar{u}).\ea$$
Since $\L(\bar{x},\bar{u}, 1, \bar{\lambda}_{\I}, \bar{\lambda}_{E}, \bar{\lambda}_{I})= F(\bar{x},\bar{u})$ assertion {\rm(i)} follows from \eqref{mmdammiirnr}. Assertion {\rm(ii)} is a direct consequence of  Theorem \ref{Teoidentificacion} and the fact that for convex problems the set of Lagrange multipliers $\Lambda_{L}(\bar{x}, \bar{u})$ does not depend on $(\bar{x}, \bar{u})$ (see e.g. \cite[Theorem 3.6]{BonSha}). \end{proof}
\section{Some sensitivity results}\label{sens}
In this section we take advantage of the Lagrange multiplier interpretation of the adjoint state $(p,q)$ in order to obtain some sensitivity results for the optimal cost when the problem dynamics and final constraints are perturbed.  We will first consider general convex problems and linear perturbations of the dynamics. Next, we study in detail the case of Linear Quadratic (LQ) stochastic problems and the mean variance portfolio selection problem, where the perturbations are performed also in the matrices multiplying  the state and control variables. We shall study these problems separately, since although they belong to a same family, their specific structures mean that we need to employ slightly different arguments and  assume different hypotheses. In any case 
a stability result for the solutions of the parameterized problems is needed and  will be a consequence of the following result:\smallskip
\begin{proposition}\label{qmmqemqmndndndndaaa} The following assertions hold: \smallskip\\
{\rm(i)} Let $x^{k} \in \ito$ be a sequence converging weakly to $x\in \ito$. Then $x^k$ converges weakly to $x$ in $(L^{2,2}_{\FF})^{n}$ and for all $t\in [0, T]$ we have that  $x^{k}(t)$ converges weakly to $x(t)$ in $(L^{2}_{\F_{t}})^{n}$.\smallskip\\
{\rm(ii)} Let $x_0^k\in \RR^{n}$, $A^{k} \in (L^{\infty,\infty}_{\FF})^{n\times n}$, $(C^{j})^{k} \in (L^{\infty,\infty}_{\FF})^{n\times n}$,  $\xi_{1}^{k} \in (L^{2,2}_{\FF})^{n}$, $(\xi_{2}^{j})^{k} \in (L^{2,2}_{\FF})^{n}$ {\rm(}$j=1,\hdots, d${\rm)}. Suppose that $(x_{0}^{k}, A^{k}, (C^{j})^{k})$ converges strongly to $(x_{0}, A, C^{j})$ and that  $(\xi_{1}^{k}, (\xi_{2}^{j})^{k})$ converges weakly to $(\xi_{1}, \xi_{2}^{j})$. Then, the solutions $x^{k}$ of 
$$\ba{rcl} \dd x^{k}(t) &=& \left[ A^{k}(t) x^{k}(t) + \xi_{1}^{k}(t)\right]\dd t+ \sum_{j=1}^{d} \left[ (C^{j})^{k}(t) x^{k}(t)+ (\xi_{2}^{j})^{k}(t)\right]\dd W^{j}(t), \\[4pt]
			x^{k}(0)&=& x_{0}^{k},\ea$$
converge weakly in $\I^n$ to the solution $x$ of 			
\be\label{mmnnndbbbfbfbbfvvvvfvss}\ba{rcl} \dd x(t) &=& \left[ A(t) x(t) + \xi_{1}(t)\right]\dd t+ \sum_{j=1}^{d} \left[ C^{j}(t) x(t)+ \xi_{2}^{j}(t)\right]\dd W^{j}(t), \\[4pt]
			x(0)&=& x_{0}.\ea\ee
{\rm(iii)} Let $D^{k} \in  (L^{\infty,\infty}_{\FF})^{n\times n}$, $(E^{j})^{k}    \in (L^{\infty,\infty}_{\FF})^{n\times n}$  {\rm(}$j=1, \hdots, d${\rm)},    $\xi_{3}^{k}\in (L^{2,2}_{\FF})^{n}$ and  $\xi_{4}^{k} \in(L^{2}_{\F_{T}})^{n}$.  Suppose that $(D^{k},(E^{j})^{k})$ converges strongly to $(D,E^{j})$ and $(\xi_{3}^{k},\xi_{4}^{k})$ converge weakly to $(\xi_{3},\xi_{4})$. Then,     the solution $(p^k,q^k)$ of 
\be\label{memqmemqmdnndndndnss}
\ba{rcl}
\dd p^{k}(t) &=& \left[ D^{k}(t) p^{k}(t) + \sum_{j=1}^{d}(E^{j})^{k}(t) (q^{j})^{k}(t) + \xi_{3}^{k}(t)\right] \dd t + q^{k}(t)\dd W(t), \\[4pt]
	p^{k}(T)&=& \xi_{4}^{k}.
\ea
\ee
converges weakly in $\I^n\times (L^{2,2}_{\FF})^{n\times d}$ to the solution $(p,q)$ of 
\be\label{memqmemqmdqeeenndndndnss}
\ba{rcl}
\dd p(t) &=& \left[ D(t) p(t) + \sum_{j=1}^{d}E^{j}(t) q^{j}(t) + \xi_{3}(t)\right] \dd t + q(t)\dd W(t), \\[4pt]
	p(T)&=& \xi_{4}.
\ea
\ee
\end{proposition}
\begin{proof}  Assertion {\rm(i)} follows directly from Lemma \ref{meqwnenenqneqbebqbs} and the fact that  $(L^{2,\infty}_{\FF})^{n}$ is continuously embedded in  $(L^{2,2}_{\FF})^{n}$ and $(L^{2}_{\F_{t}})^{n}$, for all $t\in [0,T]$.
Let us prove assertion {\rm(ii)}. Since $|x_{0}^{k}|$, $\|A^{k}\|_{\infty, \infty}$, $\|(C^{j})^{k}\|_{\infty, \infty}$, $\|(D^{j})^{k}\|_{\infty, \infty}$, $\|\xi_{1}^{k}\|_{2,2}$ and $\|(\xi_{2}^{j})^{k}\|_{2,2}$ are bounded,  by the classical proof for the stability of linear SDEs (see e.g. \cite[Chapter 6, Section 4]{YongZhou}), we have that $\| x^{k}\|_{2,\infty}$ is uniformly bounded in $k$. Therefore for any subsequence there exists $\hat{x} \in (L^{2,2}_{\FF})^{n}$ such that for a further subsequence $x^{k}$ converges weakly in $(L^{2,2}_{\FF})^{n}$ to $\hat{x}$. Using that $A^{k} x^{k}$, $(C^{j})^{k}x^{k}$ converge weakly in $(L^{2,2}_{\FF})^{n}$ to $A \hat{x}$, $C^{j}\hat{x}$, respectively, we see that $x^{k}$ converges weakly in $\I^{n}$ to 
$$ \tilde{x}(\cdot):= x_{0} + \int_{0}^{\cdot} \left[ A(t) \hat{x}(t) + \xi_{1}(t)\right]\dd t+ \sum_{j=1}^{d} \int_{0}^{\cdot} \left[ C^{j}(t) \hat{x}(t)+ \xi_{2}^{j}(t)\right]\dd W(t).$$
By {\rm(i)} we have that $ \tilde{x}= \hat{x}$ and since \eqref{mmnnndbbbfbfbbfvvvvfvss} has a unique   solution (and so independent of the given subsequence) the result follows. In order to prove {\rm(iii)}, we argue in a similar manner. Note that since $(\xi_{3}^{k}, \xi_{4}^{k})$ is bounded in $(L^{2,2}_{\FF})^{n}\times (L^{2}_{\F_{T}})^{n}$ and  $\|D^{k}\|_{\infty, \infty}$, $\|(E^{j})^{k}\|_{\infty, \infty}$ are bounded, following the lines of the proof \cite[ Chapter 7, Theorem 2.2]{YongZhou})  we obtain that $\| p^{k}\|_{2,\infty}+ \sum_{j=1}^{d} \| (q^{j})^{k}\|_{2,2}$ is uniformly bounded in $k$.   So for any subsequence there exists $(\hat{p},\hat{q}) \in (L^{2,2}_{\FF})^{n}\times  (L^{2,2}_{\FF})^{n\times d}$ such that, except for some further subsequence, $(p^{k}, q^{k})$ converge to $(\hat{p}, \hat{q})$ weakly in $(L^{2,2}_{\FF})^{n} \times (L^{2,2}_{\FF})^{n\times d}$.   Since $D^{k} p^{k}$ and $(E^{j})^{k} (q^{j})^{k}$ converge weakly in $(L^{2,2}_{\FF})^{n}$ respectively to $D \hat{p}$, $E^{j} \hat{q}^{j}$, we easily obtain that $p^{k}$ converges weakly in $\I^{n}$ to
\be\label{mmdandnanbbbbbdasad} \tilde{p}(\cdot):=\tilde{p}(0)+ \int_{0}^{\cdot} \left[ D(t) \hat{p}(t) + \sum_{j=1}^{d}E^{j}(t) \hat{q}^{j}(t) + \xi_{3}(t)\right] \dd t + \int_{0}^{\cdot} \hat{q}(t) \dd W(t),\ee
where 
$$ \tilde{p}(0):= \EE\left( \xi_4 -\int_{0}^{T} \left[ D(t) \hat{p}(t) + \sum_{j=1}^{d}E^{j}(t) \hat{q}^{j}(t) + \xi_{3}(t)\right] \dd t \right).$$
By  {\rm(i)} we obtain that $\tilde{p}=\hat{p}$, and $\hat{p}(T)=\xi_{4}$ using that $p^{k}(T)=\xi_{4}^{k}$ converges  weakly in $(L^{2}_{\F_{T}})^{n}$ to $\xi_{4}$. From this fact and \eqref{mmdandnanbbbbbdasad}, we have that $(\hat{p}, \hat{q})$ solves \eqref{memqmemqmdqeeenndndndnss}. Finally, since the solution of \eqref{memqmemqmdqeeenndndndnss} is unique, the result follows.
\end{proof}

\subsection{Convex problems and  linear  perturbations of the dynamics}\label{sens1}
Let us define the perturbation space $\P_{1}:= \RR^{n} \times (L^{2,2}_{\FF})^{n}  \times (L^{2,2}_{\FF})^{n \times d}$ and let $P:=( x_0,  \hat{f},   \hat{\sigma}) \in \P_{1}$.    We  consider the problem  \small
$$\ba{l}   \ds \inf_{(x,u) \in \ito \times (L^{2,2}_{\FF})^{m}} \EE\left( \int_{0}^{T} \ell(t,\omega,x(t), u(t)) \dd t + \Phi(\omega,x(T))\right)\\[4pt]
\mbox{s.t. } \hspace{0.4cm}\left\{\ba{rcl} \dd x (t)&=&  [f(t,\omega, x(t),u(t))+\hat{f}(t,\omega)]  \dd t +[\sigma(t,\omega, x(t),u(t))+ \hat{\sigma}(t,\omega)] \dd W(t),  \\[4pt]
								    x(0)& =& x_0, \\[4pt]
								     u  & \in & \U. \ea\right.
\ea \eqno(P_{1,P})$$ \normalsize
\normalsize
We suppose that  $(\ell,\Phi, f, \sigma)$ satisfy assumptions  \textbf{(H1)}-\textbf{(H2)} in Section \ref{mwmnnsndnadawwwapapapapas} and    $\U$ is given by \eqref{qmemqmenandadds}. In addition, we will need the following convexity assumption:\smallskip\\
\textbf{(H3)}  For almost all $(t,\omega) \in [0,T] \times \Omega$ (respectively $\omega \in \Omega$),  the function $\ell(t,\omega, \cdot, \cdot)$ (respectively $\Phi(\omega, \cdot)$) is    convex. Moreover,  we assume that a.s. in $[0,T] \times \Omega$ the functions
$f(t,\omega, \cdot, \cdot)$ and $\sigma(t,\omega, \cdot, \cdot)$ are affine. \smallskip 

We define {\it the value function} $v:  \P_{1} \to \RR \cup \{-\infty\}$ for the function that associates to $P$ the optimal cost for problem $(P_{1,P})$.  Note that under  \textbf{(H1)}-\textbf{(H2)} the feasible set for $(P_{1,P})$ is not empty, and therefore $v$ is well defined.
The following result is a consequence  of a Theorem due to R.T. Rockafellar  (see \cite{Rockafellarbook1974}), the Lagrange multiplier interpretation of   weak-Pontryagin multipliers in Theorem \ref{Teoidentificacion} and classical results in infinite dimensional optimization (see e.g. \cite{maurerzowe79, MR0406529,ZoweKurc79}). 
\begin{theorem}\label{sensibilidad1} Assume  \textbf{(H1)}-\textbf{(H3)} and  that for $P \in \P_{1}$  problem  $(P_{1,P})$ admits at least one solution. Then, there  exists   $(\bar{p}, \bar{q}) \in \ito \times  (L^{2,2}_{\FF})^{n \times d}$  such that  for every solution $(\bar{x},\bar{u})$ of $(P_{1,P})$, the pair $(\bar{p}, \bar{q})$ is the unique weak-Pontryagin multiplier associated to $(\bar{x}, \bar{u})$. Moreover,    the value function $v$ is continuous at $P$, Hadamard and Gâteaux directionally differentiable  at $P$ and its directional derivative $Dv(P; \cdot): \P_{1} \to \RR$  is given by  
\be\label{mrwmmrnnsnsnsssaaaaasss} Dv(P; \Delta P)=   \bar{p}(0)^{\top} \Delta x_0 + \EE\left( \int_{0}^{T} \bar{p}(t)^{\top}  \Delta f(t) \dd t\right) + \EE\left( \int_{0}^{T} \mbox{{\rm tr}}\left[\bar{q}(t)^{\top}  \Delta \sigma (t) \right] \dd t\right),\ee \normalsize
for all  $\Delta P =(\Delta x_0, \Delta f, \Delta \sigma) \in \P_1$.
\end{theorem}
\begin{proof}  Let us write the problem $(P_{1,P})$ as 
$$ \inf_{(x,u)\in \ito\times \control} F(x,u) + \chi_{\U}(x,u) \; \; \mbox{subject  to } G(x,u) + P= 0, $$
where $\chi_{\U}: \ito \times \control \to \RR\cup \{+\infty\}$ is the convex, proper, l.s.c. function defined as $\chi_{\U}(x,u)=0$ if $u\in \U$ and $+\infty$ otherwise and  
$$ G(x,u)(\cdot):=   \int_{0}^{\cdot}f(t,\omega, x(t),u(t))  \dd t +\int_{0}^{\cdot}\sigma(t,\omega, x(t),u(t))  \dd W(t)-x(\cdot).$$ 
For every $(x,u) \in \ito \times \control$ and $v\in \control$, Lemma \ref{wrnwnrnwnrnqbqbqbq} implies that $ D G(x,u)(\cdot, v)$ is surjective. Therefore, the following regularity condition is trivially satisfied (see e.g.  \cite[Section 3.2]{bstour})
\be\label{regularidadcasocombinado} 0 \in \mbox{int}\left\{  G(x,u)+ P+ DG(x,u)\left( \ito \times \U \right) \right\}.\ee
Thus, by classical results in convex optimization (see e.g. \cite[Section 4.3.2, Example 4.51]{BonDUNOD06} or \cite[Section 2.5]{BonSha}) $(x,u)$ is a solution of $(P_{1,P})$ iff there exists $\lambda \in \ito$ such that 
\be\label{condicionoptimalidadconsubdiferencial}(0,0) \in \partial_{(x,u)} (F(x,u)+\chi_{\U}(x,u))+   D G(x,u)^{*} \lambda .\ee
Since $F$ is differentiable in $\ito \times \control$, in particular it is continuous in $\ito \times \U$, and so (see e.g. \cite[Remark 2.170]{BonSha}) \small
$$ \partial_{(x,u)} (F(x,u)+ \chi_{\U}(x,u))= \partial_{(x,u)}F(x,u)  +\partial_{(x,u)}  \chi_{\U}(x,u) \subseteq (D_{x} F(x,u),D_{u}F(x,u))+ \{0\} \times N_{\U}( u),$$ \normalsize
where $N_{\U}( u):=\{ v^{*} \in \control \; ; \; \langle v^{*}, v-u  \rangle_{L^{2}} \leq 0, \; \; \forall \; v \in \U\} $ is the normal cone to $\U$ at $u$. Using that $D G(x,u)^{*} \lambda= (D_{x} G(x,u)^{*} \lambda, D_{u} G(x,u)^{*} \lambda)$, we obtain with \eqref{condicionoptimalidadconsubdiferencial}
$$ (0,0) \subseteq  (D_{x} F(x,u),D_{u}F(x,u))+ \{0\} \times N_{\U}( u)+ (D_{x} G(x,u)^{*} \lambda, D_{u} G(x,u)^{*} \lambda),$$
which is equivalent to 
\be\label{aqmemqmesssa}  D_{x}  \L(x, u, \lambda)=0  \hspace{0.3cm} \mbox{and }  \hspace{0.3cm} D_{u}  \L(x,u,\lambda)(v-u)\geq 0 \hspace{0.5cm} \forall v \in  \U.\ee
Therefore, $\lambda \in \Lambda_{L}(x,u)$ and by Theorem \ref{Teoidentificacion} and the convexity of the associated Hamiltonian we have that $(\bar{p}, \bar{q}):=(\lambda_{1}, \lambda_{2})$ is weak-Pontryagin multiplier. 
%
Now, let $\lambda_{\I}^{1}$,  $\lambda_{\I}^{2}\in \Lambda_{L}(x, u)$. By the first equation  in \eqref{aqmemqmesssa}, we get that 
 $$ \left\langle \left(D_{x}G(x,u) \right)^{*}( \lambda_{\I}^{1}-  \lambda_{\I}^{2}), z \right\rangle_{\I} =0 \; \; \forall \; z\in \ito, \hspace{0.3cm} \mbox{or   } \; \left(D_{x}G(x, u) \right)^{*}( \lambda_{\I}^{1}-  \lambda_{\I}^{2})=0.$$
 Since, by Lemma  \ref{wrnwnrnwnrnqbqbqbq},  $D_{x}G(x,u): \ito \mapsto \ito$  is surjective we get that    $D_{x}G(x,u)^{*}$ is injective, which implies that $ \lambda_{\I}^{1}=  \lambda_{\I}^{2}$ and by Theorem \ref{Teoidentificacion} the weak-Pontryagin multiplier is unique. The independence of the set $\Lambda_{L}(\cdot)$ over the set of solutions of $(P_{1,P})$ is a consequence of  Corollary \ref{memmndnaaaaasssasasasasasasasa}{\rm(ii)}. Finally, the continuity, the Gâteaux and  Hadamard differentiability of $v$ and  expression \eqref{mrwmmrnnsnsnsssaaaaasss} for $Dv(P; \Delta P)$ are a direct translation of  \cite[Theorem 17]{Rockafellarbook1974}   using the uniqueness of the Lagrange multiplier.
 \end{proof} \smallskip 
 
In the following remark we underline some simple consequences of Theorem \ref{sensibilidad1}:
\begin{remark}\label{mqmennnnaasasasa}
{\rm(i)} The gradient of $v$ at $P$, i.e. the Riesz representative of the bounded linear application $Dv(P; \cdot)$, is given by 
$$\bar{p}(0)+ \int_{0}^{\cdot} \bar{p}(t) \dd t + \int_{0}^{\cdot} \bar{q}(t) \dd W(t).$$
{\rm(ii)} It is well known {\rm(}see e.g. {\rm \cite[Section 2.2]{BonSha}} and the references therein{\rm)} that   for real-valued functions defined on  finite dimensional spaces,  Gâteaux differentiability together with  Hadamard differentiability imply Fr\'echet differentiability. Therefore,    if the perturbations for problem $(P_{1,P})$  are finite dimensional, then $v$ is Fr\'echet differentiable at $P$. This is the case,  for example, if  the initial condition is perturbed and/or the perturbations of the dynamics have the form $\Delta f(t,\omega)=\xi_{0}(t,\omega)A_{0}$,   $(\Delta \sigma (t,\omega))^{j}=\xi_{j}(t,\omega) A_{j}$ with $\xi_{0}$, $\xi_{j} \in (L^{\infty,\infty}_{\FF})^{n\times n}$ {\rm(}$j=1,\hdots,d${\rm)} being fixed, and $A_{0}$, $A_{j} \in \RR^{n}$ being the perturbation parameters.  In fact, defining the new states
$$ \dd y_{0}= 0, \; \; \mbox{for $t\in [0,T]$}, \; \; y_{0}(0)= A_{0}, \; \;  \dd y_{j}= 0, \; \; \mbox{for $t\in [0,T]$}, \;  \; y_{j}(0)= A_{j} \; \; \mbox{for $j=1,\hdots, d$},$$
the new dynamical system is affine w.r.t. $\left( x,(y_0, y_j)\right)$ and the perturbations are performed over the initial condition.  Let us point out that the Fr\'echet differentiability of the value function under finite-dimensional perturbations in our convex framework can also be deduced using {\rm\cite[Chapter 5, Corollary 4.5]{YongZhou}}.
 \smallskip\\
{\rm(iii)} Suppose that the nominal problem is deterministic (and thus $\bar{q}=0$) and only the $\dd W(t)$ part of the dynamics is perturbed, i.e.  $\Delta x_0=0$, $\Delta f \equiv 0$. Then, by \eqref{mrwmmrnnsnsnsssaaaaasss} we directly obtain that $Dv(P; \Delta P)=0$. This fact was already observed by  Loewen {\rm \cite{Loewen87}} for finite dimensional perturbations.  \smallskip\\
{\rm(iv)} A close look at the proof Theorem \ref{sensibilidad1} shows that even if $\ell(\omega, t, \cdot, \cdot)$ and $\Phi(\omega, \cdot)$ are not convex, we can apply the  abstract optimization results {\rm(}see e.g. {\rm\cite[Section 3.1]{BonSha}}{\rm)} in order to derive existence and uniqueness of a Lagrange multiplier at a local solution $\bar{u}$.  More precisely, using \eqref{regularidadcasocombinado} it is possible to show {\rm(}see {\rm\cite[Lemma 3.7]{BonSha})} that if $(x,u)$ is a solution of    problem $(P_{1,P})$ then $(z,v)=(0,0)$ is a solution of 
$$ \inf_{(z,v)\in \ito \times \control} D F(x,u)(z,v) \; \; \mbox{such that } \; \;  DG(x,u)(z,v)=0, \; \;   v \in T_{\U}(u), \; \; \eqno(LP)$$
{\rm(}where $T_{\U}(u)$,  defined as the closure in $\control$ of $\bigcup_{\tau>0} \tau^{-1}(\U-u)$, is the tangent cone to $\U$ at $u$, see {\rm \cite[Proposition 2.55]{BonSha})}. Problem $(LP)$ is a convex one and we can proceed exactly as in the proof of Theorem \ref{sensibilidad1} in order to show the existence and uniqueness of a Lagrange multiplier $\lambda$ at $(0,0)$. It is easy to see that $\lambda$ is a Lagrange multiplier at $(0,0)$ for problem $(LP)$ iff $\lambda$ is a Lagrange multiplier at $(x,u)$ for problem  $(P_{1,P})$. Therefore,  by Theorem \ref{Teoidentificacion}  this argument provides a simple proof  of the existence of weak-Pontryagin multipliers for stochastic problems with non-convex cost and linear dynamics. Let us point out that it is not clear that the general result of {\rm \cite{Peng90}} for the case of nonlinear dynamics, even in the form of a weak-Pontryagin principle, can be derived with the Lagrange multipliers method. In fact, the main issue is the apparent lack of  $C^{1}$ differentiability of  $G(x,u)$ in the non-affine case {\rm(}see Remark \ref{admamdmamnnnssssaaaa}{\rm)}. 
\end{remark}
\normalsize

We consider now the case of {\it final state constraints  without control constraints}\footnote{Actually we can handle also control   and final state constraints simultaneously under a suitable qualification condition  (see \cite[Section 3.2]{bstour}). However, for the sake of simplicity we preferred to state the results for both types of constraints separately.}. We set as parameter    set the space $\P_{2}:= \RR^{n} \times (L^{2,2}_{\FF})^{n}  \times (L^{2,2}_{\FF})^{n \times d} \times \RR^{n_E}\times \RR^{n_{I}}$. Let $P:=( x_0,  \hat{f},   \hat{\sigma}, \delta_{E}, \delta_{I}) \in \P_{2}$  and   consider the problem 
\small
$$\ba{l}   \ds \inf_{(x,u) \in \I^{n} \times (L^{2,2}_{\FF})^{m}} \EE\left( \int_{0}^{T} \ell(t,\omega,x(t), u(t)) \dd t + \Phi(\omega,x(T))\right)\\[4pt]
\mbox{s.t. } \hspace{0.4cm}\left\{\ba{rcl} \dd x (t)&=&  [f(t,\omega, x(t),u(t))+\hat{f}(t,\omega)]  \dd t  \\[4pt]
								\;         & \; & +[\sigma(t,\omega, x(t),u(t))+ \hat{\sigma}(t,\omega)] \dd W(t),  \\[4pt]
								    x(0)& =& x_0, \\[4pt]
								    \EE\left(\Phi_{E}^{i}(\omega, x(T)) \right)  &=&-\delta^{i} \hspace{0.5cm} \mbox{for all } i=1,\hdots,n_{E}, \\[4pt]
								    \EE\left(\Phi_{I}^{j}(\omega,x(T)) \right)  &\leq &-\delta^{j} \hspace{0.5cm} \mbox{for all } j=1,\hdots,n_{I}. \ea\right.
\ea \eqno(P_{2,P})$$ \normalsize
We will assume that:\smallskip\\
\textbf{(H4)} For almost all $\omega \in \Omega$  the functions $\Phi_{E}^{i}(\omega,\cdot)$ ($i=1,\hdots, n_{E}$) are affine and  $\Phi_{I}^{j}(\omega,\cdot)$ ($j=1,\hdots, n_{I}$) are convex.  \vspace{0.4cm}

The proof of the following result follows the same lines as those in the proof of Theorem \ref{sensibilidad1} and therefore is omitted.  Recall that  $G$ is defined in \eqref{meqmemnnrnnrnr} and  $G_{E}$, $G_{I}$ are defined in \eqref{mamsmnnrnnrnrs}.
\begin{theorem}\label{madnandeueueuaasas} Assume   \textbf{(H1)}-\textbf{(H4)} and  that for $P \in \P_{2}$  problem  $ (P_{2,P})$ admits at least one solution  $(\bar{x}, \bar{u})$. Suppose in addition  that the following Slater constraint qualification condition at $(\bar{x},\bar{u})$ holds
$$\small
 \left.\ba{l} {\rm(i)} \; \; \left(DG(\bar{x}, \bar{u}), DG_{E}(\bar{x})\right): \ito \times \control \mapsto \ito  \times \RR^{n_{E}} \; \; \mbox{is surjective and }\\[4pt]
				{\rm(ii)} \; \; \exists \; (\hat{z}, \hat{v}) \in \ito \times \control \; \;  ; \;  G(\hat{z},  \hat{v})=0, \; \;  G_{E}(\hat{z},  \hat{v})=0, \; \ G_{I}^{j}(\hat{z})<0 \; \, \forall \; j=1,\hdots,, n_I.  \ea \right\}  \eqno(S)$$\normalsize
Then, the set  of weak-Pontryagin multipliers $\Lambda_{wP}(\bar{x},\bar{u})\subset \ito\times\RR^{n_E+n_I}$ at any solution $(\bar{x}, \bar{u})$  is a non-empty, weakly compact set, which is independent of the solution $(\bar{x}, \bar{u})$.    Moreover, the value function $v$ is continuous  at $P$, Hadamard  directionally differentiable at $P$  and its directional derivative $Dv(P; \cdot): \P_{2} \to \RR$ is given by 
 \small	
$$\ba{ll}Dv(P; \Delta P)=  \ds \max_{(p,q, \lambda_{E}, \lambda_{I}) \in \Lambda_{wP}(\bar{x}, \bar{u})}  &\left\{ p(0)^{\top} \Delta x_0 +   \EE\left( \int_{0}^{T} p(t)^{\top}  \Delta f(t) \dd t\right) +  \EE\left( \int_{0}^{T}  \mbox{{\rm tr}}\left[q(t)^{\top}  \Delta \sigma (t)\right] \dd t\right)\right. \\[4pt]
\; & +\left. \lambda_{E}^{\top} \Delta \delta_{E}+ \lambda_{I}^{\top} \Delta \delta_{I} \right\},\ea$$ \normalsize
for all  $ \; \Delta P =(\Delta x_0, \Delta f, \Delta \sigma, \Delta \delta_{E},  \Delta \delta_{I} ) \in \P_2$.
\end{theorem}
\begin{remark}\label{mmdnsnqwwwwwwwqeodoff}{\rm(i)} Note that if no inequality constraints are present (which can be written as $n_{I}=0$), the qualification condition for $(P_{2,P})$ is given by $(S)${\rm(i)}. In this case, as in Theorem \ref{sensibilidad1}, we get the uniqueness of the multiplier and thus $v$ is also Gâteaux differentiable at $P$.

{\rm(ii)}  Since $(G,G_{E})$ is affine and $G^{j}_{I}$ {\rm(}$j=1,\hdots, n_I${\rm)} are  convex, we have that the Slater condition $(S)$ is equivalent to the following Mangasarian-Fromovitz condition
$$\small
 \left.\ba{l} {\rm(a)} \; \; \left(DG(\bar{x}, \bar{u}), DG_{E}(\bar{x})\right): \ito \times  \control \mapsto \ito  \times \RR^{n_{E}} \; \; \mbox{is surjective and }\\[4pt]
				{\rm(b)} \; \; \exists \; (\hat{z}, \hat{v}) \in \mbox{{\rm Ker}}DG(\bar{x}, \bar{u}) \cap  \mbox{{\rm Ker}}DG_{E}(\bar{x}) \; \; \mbox{such that } DG_{I}^{j}(\bar{x})\hat{z}<0 \; \, \forall \; j=1,\hdots,, n_I.  \ea \right\}  \eqno(MF)$$\normalsize
Condition $(MF)$ has been stated in the literature {\rm(}see e.g. {\rm \cite{bonnanssilva12})} for the reduced optimal control problem $(\mathcal{SP'})$. More precisely,   for $v\in \control$ let $z[v]\in \ito$ be defined by the equation $DG(\bar{x},\bar{u})(z,v)=0$. Since this is a standard linear SDE in the variable $z$, under our assumptions, we get that $z[v]$ is well defined.  We check then that $(MF)$ is equivalent to 
$$\small
 \left.\ba{l} {\rm(a')} \; \;   v \in  \control \to DG_{E}(\bar{x})z[v] \in \RR^{n_{E}}  \hspace{0.2cm}  \mbox{is surjective and }\\[4pt]
				{\rm(b')} \; \; \exists \; \hat{v}  \in \control \; \; \mbox{such that } \;   DG_{E}(\bar{x})z[v]=0 \; \; \mbox{and } DG_{I}^{j}(\bar{x})z[v]<0 \; \, \forall \; j=1,\hdots,, n_I.  \ea \right\} $$\normalsize
\end{remark}
\subsection{Multiplicative perturbations in the Linear Quadratic  framework}\label{sens2}
In this part we adopt the framework of {\it unconstrained}  Linear Quadratic (LQ) stochastic control problems with random coefficients (see e.g \cite{Bismut76a,ChenYong01,Tang,YongZhou} and the references therein). More precisely, let us consider the problem
\small
$$\ba{l}  \ds \inf_{(x,u) \in \I^{n} \times (L^{2,2}_{\FF})^{m}}  F(x,u) := \frac{1}{2}\EE\left( \int_0^T \left[ x(t)^{\top}Q(t)x(t) + u(t)^{\top}N(t)u(t) \right] \dd t +x(T)^{\top}Mx(T)\right) \\[6pt]

\mbox{s.t. } \left\{\ba{l}  \dd x(t)= [A(t)x(t)+B(t)u(t)+e(t)]\dd t+ \sum\limits_{j=1}^d [C^j(t)x(t)+D^j(t)u(t)+f^j(t)]\dd W^j(t),\\
x(0)  \;   = x_0.\ea \right. \ea\eqno(P_{3,P})
$$\normalsize
We shall view $P = (x_0, A, B, C^j, D^j, e, f^{j})$ ($j=1,\hdots, d$) as parameters for the problem $(P_{3,P})$. Thus, we consider as  parameter space
$$\mathcal{P}_3= \RR^{n} \times (L^{\infty,\infty}_{\FF})^{n\times n}\times (L^{\infty,\infty}_{\FF})^{n\times m} \times (L^{\infty,\infty}_{\FF})^{(n\times n)\times d} \times (L^{\infty,\infty}_{\FF})^{ (n\times m) \times d }\times (L^{2,2}_{\FF})^n\times   (L^{2,2}_{\FF})^{n\times d }.$$ 
It is well known (see \cite[Theorem 2.1]{Bismut76a}) that given $P \in \P_{3}$ and $u \in \control$ the linear  SDE in $(P_{3,P})$ admits a unique solution in $\I^{n}$. We will also need the following result:
\begin{lemma} 
\label{regularidadLQ}
The constraint function
$G:\ito \times \control \times \mathcal{P}_3 \mapsto \I^{n}$ defined by:
$$
\ba{lll}
G(x,u,P)&:=& - x(\cdot) + x_0 +\int_0^{\cdot} [A(t)x(t)+B(t)u(t)+e(t)]\dd t \\
\; & \; & + \int_0^{\cdot} \sum\limits_{j=1}^d [C^j(t)x(t)+D^j(t)u(t)+f^j(t)]\dd W^j(t),
\ea
$$
is continuously Fr\`echet  differentiable. Furthermore, $D_{(x,u)}G(x,u,P)$ is onto. 
\end{lemma}

\begin{proof}
That $G$ is well defined is a simple application of Lemma \ref{meqwnenenqneqbebqbs}. Following the lines of the proof of Lemma \ref{wrnwnrnwnrnqbqbqbq}  we have that $G$ is Gâteaux differentiable at any $(x,u, P) \in \I^{n}\times \control \times \P_{3}$ and for every  $(x',u') \in \I^n \times \control$ and  $P'=(x_{0}',A',B',\{(C^j)'\},\{(D^j)'\},e',\{(f^j)'\})\in \P_3$ we have that    
$$ 
\ba{lll}
DG(x,u,P)(x',u',P')&=&  \int_{0}^{\cdot} [Ax'+Bu'+e']\dd t+ \int_{0}^{\cdot}\sum\limits_{j=1}^d [C^{j}x'(t)+D^ju'+(f')^j]\dd W^j(t)\\
\; & \; & + \int_{0}^{\cdot} [A'x+B'u]\dd t+\int_{0}^{\cdot} \sum\limits_{j=1}^d [(C^j)'X+(D^j)'u]\dd W^j(t)+x_0'-x'(\cdot).
\ea
$$
Thus, for every $(x_{1}, u_{1})$, $(x_{2}, u_{2})\in \ito \times \control$  and $P_{1}$, $P_{2}\in \P_{3}$ we have that 
$$ \|DG(x_1,u_1,P_1)(x',u',P')- DG(x_2,u_2,P_2)(x',u',P')\|_{\I}^{2} $$
is given by  
$$
\ba{l}
\EE\left( \int_{0}^{T}\left| (A_1-A_2)x'+(B_1-B_2)u'+A'(x_1-x_2)+B'(u_1-u_2)\right|^2   \dd t \right ) +\\
\EE \left ( \sum_{j=1}^{d}\int_{0}^{T} \left| (C_1^j-C_2^j)x'+(D_1^j-D_2^j)u'+(C^j)'(x_1-x_2)+(D^j)'(u_1-u_2)\right|^2 \dd t \right).
\ea
$$
Therefore, if $\|P'\|=1$, we find that $ \|DG(x_1,u_1,P_1)P'- DG(x_2,u_2,P_2)P'\|_{\I}^{2}$  is bounded by     \footnotesize
$$
c\left(  \| x_{1}-x_{2}\|_{\I}^{2}+\| u_{1}-u_{2}\|_{2}^{2} + \| A_{1}- A_{2}\|_{\infty}^{2}+ \| B_{1}- B_{2}\|_{\infty}^{2}+ \sum_{j=1}^{d}\left[\| C^j_{1}- C^j_{2}\|_{\infty}^{2}+ \| D^j_{1}- D^j_{2}\|_{\infty}^{2}\right] \right),
$$
\normalsize
for some $c>0$, where we used Lemma \ref{meqwnenenqneqbebqbs} to make $\|\cdot\|_{\I}$ appear. Thus, $G$ is G\^ateaux differentiable with a continuous directional derivative, and so   $G$ is indeed Fr\`echet continuously differentiable. The surjectivity of $D_{(x,u)}G(x,u,P)$ follows from  Lemma \ref{wrnwnrnwnrnqbqbqbq}.
%
%
\end{proof}\vspace{0.3cm}

 We make the following convexity assumption:\smallskip\\
\textbf{(H5) } The matrix processes $Q:[0,T]\times\Omega\mapsto\RR^{n\times n}$, $N:[0,T]\times\Omega\mapsto\RR^{m\times m}$ are essentially bounded and progressively measurable, whereas the matrix $M:\Omega\mapsto\RR^{n\times n}$ is essentially bounded and $\F_T$-measurable. In addition $Q$, $N$ and $M$ are a.s. non-negative  symmetric  matrices and further there exists $\delta>0$ such that  $N\succeq\delta I$.\vspace{0.25cm}

By  \cite[Theorem 3.1]{Bismut76a} we have that  under \textbf{(H5)} problem $(P_{3,P})$ admits a unique solution $(x[P],u[P])$.
%
%
%
Moreover, by \cite[Theorem 3.2]{Bismut76a} (or  Theorem \ref{sensibilidad1})  we obtain the existence of a unique weak-Pontryagin multiplier $(p[P], q[P]) \in \ito \times (L^{2,2}_{\FF})^{n\times d}$ such that
\be\label{qmmeqmenqneaeweqeqqqsa}
\small 
\ba{rll}
\dd  x(t)&=& [A(t) x(t)+B(t) u(t)+e(t)]\dd t+ \sum\limits_{j=1}^d [C^j(t) x(t)+D^j(t) u(t)+f^j(t)]\dd W^j(t), \\[4pt]
u(t)&=& -N(t)^{-1}\left[B(t)^{\top}p(t)+\sum_{j=1}^d D^j(t)^{\top}q^{j}(t) \right], \\[4pt]
\dd p(t) &=& -[A(t)^{\top}p(t) + \sum_{j=1}^d C^j(t)^{\top}q^{j}(t) + Q(t)x(t) ]\dd t+\sum_{j=1}^d q^{j}(t) \dd W^j(t), \\[4pt]
x(0)=x_0&,&  p(T)=Mx(T),
\ea 
\normalsize 
\ee
where we have omitted  the dependence on $P$ in order to simplify the notation. We want to obtain now an energy estimate for $(x[P], u[P],p[P],q[P])$  in terms of $P$, in the spirit of \cite[Theorem 2.2]{Tang}. Because we need to keep track of the constants that will appear (since they depend on model parameters, which we shall later vary) we prove the following Lemma:

\begin{lemma}
\label{PontTang} Under \textbf{(H5)} there exists a continuous function $\beta: \P_{3} \to \RR$ such that  
$$ \| x[P] \|_{\I}^{2} + \| u[P]\|_{2,2}^{2} + \| p[P] \|_{\I}^{2} + \sum_{j=1}^{d} \| q^{j}[P]\|_{2,2}^{2} \leq\beta(P).$$
\end{lemma}
\begin{proof} 
For notational convenience we will omit the dependence on $P$ of $(x[P], u[P], p[P], q[P])$. A close look at the classical proof for the stability of solutions to linear SDEs  (see e.g. \cite[Chapter 6, Section 4]{YongZhou})  and of linear BSDEs (see e.g. \cite[ Chapter 7, Theorem 2.2]{YongZhou})   gives that
\be\label{menqnoooasqwqwqw}\ba{rcl}
\| x \|_{2,\infty}^{2}&\leq & \kappa_{0}(P)\left(\| u \|_{2,2}^{2} + |x_0|^{2} +\| e\|_{2,2}^{2} + \sum_{j=1}^{d} \| f^{j}\|_{2,2}^{2}\right), \\[4pt]
\| p\|_{2,\infty}^{2} + \sum_{j=1}^{d} \| q^{j}\|_{2,2}^{2}  &\leq& \kappa_1(P) \EE\left (|M(T)x(T)|^2+\int_0^{T}|Q(t)x(t)|^2 \dd t \right ),\ea\ee
where \footnotesize
$$\kappa_{0}=\kappa_{0}(\|A\|_{\infty,\infty}, \|B\|_{\infty,\infty}, ,\sum_{j=1}^{d}\|C^j\|_{\infty,\infty}, ,\sum_{j=1}^{d}\|D^j\|_{\infty,\infty}), \; \; \mbox{and } \; \; \kappa_1=\kappa_1\left(\|A\|_{\infty,\infty},\sum_{j=1}^{d}\|C^j\|_{\infty,\infty}\right),$$\normalsize
are  continuous functions. Recall that for a symmetric non-negative matrix $L\in \RR^{n\times n}$ one has that $k_{L} L  \succeq L^2$ for $k_{L}$ equals the largest eigenvalue of $L$. It is easy to check that $k_{L} \leq n \max_{i, j \in \{1, \hdots, n\}} |L^{ij}|$.
Applying this we see that  
\be\label{qmmqqqqqenqnsada}\ba{rcl}\int_0^{T}|Q(t)x(t)|^2\dd t&\leq& c  \int_0^{T} x(t)^{\top} Q(t) x(t)   \dd t,\\[4pt]
|M(T)x(T)|^2&\leq& c  x(T)^{\top} M(T)x(T),\ea\ee\normalsize
where $c=n \max\{\| Q\|_{\infty,\infty},  \| M\|_{\infty}\}$. 
Now, combining Lemma \ref{usoito} and \eqref{qmmeqmenqneaeweqeqqqsa}, we get  \small
\be\label{qmmndnndndnndndnaasasada}  \EE\left( x(T)^{\top}M(T)x(T) +  \int_0^{T}\left[   x^{\top} Q x+ u^{\top}Nu\right] \dd t \right)= p(0)^{\top}x_0+\EE\left(\int_0^{T}\left [ p^{\top}e+\sum_{j=1}^{d} (q^{j})^{\top}f^j\right ]\dd t\right).\ee\normalsize
Therefore, by the second inequality in \eqref{menqnoooasqwqwqw}, \eqref{qmmqqqqqenqnsada} and \eqref{qmmndnndndnndndnaasasada} we have that 
$$\| p\|_{2,\infty}^{2} + \sum_{j=1}^{d} \| q^{j}\|_{2,2}^{2} \leq c \kappa_1 \left\{ |p(0)| |x_0|+\EE\left ( \int_0^T \left|p^{\top}e+\sum_{j=1}^{d} (q^{j})^{\top}f^j\right|\dd t\right)\right\}. $$
Using now the inequality $2 ab \leq a^2+b^2$ for all $a$, $b\in \RR$, we get that 
\be\label{mmanndnndnnpaspapspkke}\| p\|_{2,\infty}^{2} + \sum_{j=1}^{d} \| q^{j}\|_{2,2}^{2} \leq   \kappa_{2}  \EE\left(\left[\int_{0}^{T}\left( \frac{|x_0|}{T} + |e|\right) \dd t\right]^{2}+ \int_{0}^{T} \sum_{j=1}^{d} |f^{j}|^{2} \dd t\right),\ee
where $\kappa_2$ depends continuously on $c$ and $\kappa_1$ only, and so the r.h.s. is clearly a continuous function of the model parameters.  On the other hand,   by \eqref{qmmqqqqqenqnsada} we have that 
$$  \delta \| u \|_{2,2}^{2} \leq   p(0)^{\top}x_0 + \EE\left ( \int_0^T \left|p^{\top}e+\sum_{j=1}^{d} (q^{j})^{\top}f^j\right|\dd t\right). $$
Using \eqref{mmanndnndnnpaspapspkke} we obtain that $ \| u \|_{2,2}^{2}$ is bounded by a continuous function of $P$. Therefore,  from the first equation in \eqref{menqnoooasqwqwqw} we get that $ \| x \|_{2,\infty}^{2}$ is bounded by a continuous function of $P$. Thus, noting that 
$$\ba{l}   x_{1}[P]= A x[P] + B u[P]+ e \hspace{0.3cm} \mbox{and } \; \;   x_{2}^{j}[P]= C^j(t) x[P](t)+D^j(t) u[P](t)+f^j(t),  \\[4pt]
p_{1}[P]= -[A(t)^{\top}p[P](t) + \sum_{j=1}^d C^j(t)^{\top}q[P]^{j}(t) + Q(t)x[P](t) ] \hspace{0.3cm} \mbox{and } \; \; p_{2}^{j}[P]= q[P]^{j},\\[4pt]
p_0[P]= \EE\left(M x[P](T)- \int_{0}^{T} p_{1}[P](t) \dd t\right),
\ea$$ 
we obtain that $\| x[P] \|_{\I}^{2}+ \|p[P]\|_{\I}^{2}$ is bounded by a continuous function of $P$. The result follows.
\end{proof} \smallskip

We   prove now a stability result for the solutions of $(P_{3,P})$ in terms of $P$. More precisely, let  $P^{k}$ and $P\in \P_{3}$ be such that $P^{k} \to P$ as $k\to \infty$.
We have the following stability result for  $(x^{k}, u^{k},p^{k},q^{k}):= (x[P^{k}], u[P^{k}], p[P^{k}], q[P^{k}])$.
\begin{proposition}
\label{convergencia} Suppose that  \textbf{(H5)} holds true. Then, as $k\uparrow \infty$,  we have that $v(P^{k}) \to v(P)$ and $(x^{k}, u^{k},p^{k},q^{k})$ converges strongly in $\ito \times \control \times \ito \times (L^{2,2}_{\FF})^{n\times d}$ to $(\bar{x}, \bar{u}, \bar{p}, \bar{q}):=  (x[P], u[P], p[P], q[P])$.
\end{proposition}

\begin{proof} Let us first prove the convergence of the value functions.  Define $\hat{x}^{k}$  as the solution of the following SDE: \small
$$\ba{rcl}\dd  \hat{x}^{k}(t)&=& [A^{k}(t) \hat{x}^{k}(t)+B^{k}(t) \bar{u}(t)+e^{k}(t)]\dd t\\[4pt]
\; & \; &+ \sum\limits_{j=1}^d [(C^{j})^{k}(t) \hat{x}^{k}(t)+(D^{j})^{k}(t) \bar{u}(t)+(f^{j})^{k}(t)]\dd W^j(t),\\[4pt]
\hat{x}^{k}(0) &=& x_{0}^{k}.\ea
$$
\normalsize
By definition, $(\hat{x}^{k},\bar{u}) \in F(P_{3,P^{k}})$ and by the first estimate in \eqref{menqnoooasqwqwqw} we have $\hat{x}^{k}$ is bounded in $(L^{2,\infty}_{\FF})^{n}$, uniformly in $k$. Now,   $ \hat{z}^{k}: =\hat{x}^{k} - \bar{x}\in \ito$  satisfies 
$$\ba{rcl}\dd  \hat{z}^{k}(t)&=& [A(t) \hat{z}^{k}(t)+\delta^{k} A \hat{x}^{k}+  \delta^{k}B (t) \bar{u}(t)+\delta^{k}e(t)]\dd t\\[4pt]
\; & \; &+ \sum\limits_{j=1}^d [C^{j}(t)\hat{z}^{k}(t) +\delta^{k} C^{j}(t) \hat{x}^{k}+\delta^{k}D^{j}(t) \bar{u}(t)+\delta^{k} f^{j}(t)]\dd W^j(t)\\[4pt]
\hat{z}^{k}(0) &=& \delta^{k}x_{0},\ea
$$
where $\delta^{k} A:= A^{k}-A$,  $\delta^{k} B:= B^{k}-B$ and  $\delta^{k}e:= e^{k}- e$ with an analogous definition for   $\delta^{k}x_0,\delta^{k} C^{j}$, $\delta^{k}D^{j}$,  $\delta^{k} f^{j}$. By the convergence $P^{k}\to P$, the boundedness of $\hat{x}^{k}$  in $(L^{2,\infty}_{\FF})^{n}$ and classical bounds for linear SDEs  (see e.g. \cite[Chapter 6, Section 4]{YongZhou}), we get that $ \hat{z}^{k}\to 0$ in $(L^{2,\infty}_{\FF})^{n}$. This, implies that $| F(\hat{x}^{k}, \bar{u}) -F(\bar{x}, \bar{u})|$ tends to zero as $k\uparrow \infty$.  Therefore, we get
$$ v(P^{k}) \leq F(\hat{x}^{k},\bar{u}) =  F(\bar{x},\bar{u}) +o(1) = v(P)+o(1),$$
which implies that  $ \limsup_{k\uparrow \infty}    [v(P^{k})-  v(P)] \leq 0$. Analogously, if $\tilde{x}^{k}$ is the solution of 
$$\ba{rcl}\dd  \tilde{x}^{k}(t)&=& [A(t) \tilde{x}^{k}(t)+B(t) u^{k}(t)+e(t)]\dd t\\[4pt]
\; & \; &+ \sum\limits_{j=1}^d [C^{j}(t) \tilde{x}^{k}(t)+D^{j}(t)  u^{k}(t)+f^{j}(t)]\dd W^j(t),\\[4pt]
\tilde{x}^{k}(0) &=& x_{0},\ea
$$
we have that $(\tilde{x}^{k},u^{k}) \in F(P_{3,P})$. In addition,  $\tilde{z}^{k}:=  x^{k}-\tilde{x}^{k}$ satisfies
 $$\ba{rcl}\dd  \tilde{z}^{k}(t)&=& [A^{k}(t) \tilde{z}^{k}(t)+\delta^{k} A \tilde{x}^{k}+  \delta^{k}B (t) u^{k}(t)+\delta^{k}e(t)]\dd t\\[4pt]
\; & \; &+ \sum\limits_{j=1}^d [C^{j}(t)\tilde{z}^{k}(t) +\delta^{k} C^{j}(t) \tilde{x}^{k}+\delta^{k}D^{j}(t)u^{k}(t)+\delta^{k} f^{j}(t)]\dd W^j(t),\\[4pt]
\tilde{z}^{k}(0) &=& \delta^{k}x_{0}.\ea
$$ 
By Lemma \ref{PontTang} we see that $u^k$ is bounded in $(L^{2,2})^m$. So as before since $P_{k}\to P$ we get that $\tilde{x}^{k}$ is bounded in $(L^{2,\infty}_{\FF})^{n}$, and similarly obtain that $ \tilde{z}^{k} \to 0$ in $(L^{2,\infty}_{\FF})^{n}$ and so 
$| F(\tilde{x}^{k}, u^{k}) -F(x^{k}, u^{k})|\to 0$. Thus, we obtain
$$ v(P) \leq F(\tilde{x}^{k}, u^{k}) =  F(x^{k}, u^{k}) +o(1) = v(P^{k})+o(1),$$
which implies that  $ \liminf_{k\uparrow \infty}    [v(P^{k})-  v(P)] \geq 0$, proving the convergence of the value functions.  On the other hand, since $P^{k}$ converges to $P$,   Lemma \ref{PontTang}  implies the existence of $(\hat{x}, \hat{u}, \hat{p}, \hat{q})$ such that, up to some subsequence,  $(x^{k}, u^{k},p^{k},q^{k})$ converges weakly in $\ito \times \control \times \ito \times (L^{2,2}_{\FF})^{n\times d}$ to $(\hat{x}, \hat{u}, \hat{p}, \hat{q})$. By Proposition \ref{qmmqemqmndndndndaaa}, we easily get that $(\hat{x}, \hat{u}, \hat{p}, \hat{q})$ satisfies \eqref{qmmeqmenqneaeweqeqqqsa}.  By Corollary \ref{memmndnaaaaasssasasasasasasasa}, we have that $(\hat{x}, \hat{u})$ is a solution of $(P_{3,P})$, which by uniqueness implies that $(\hat{x}, \hat{u}) = (\bar{x}, \bar{u})$ and so $(\hat{p}, \hat{q}) = (\bar{p},\bar{q})$. On the other hand, using the elementary fact that for every sequences $a_k$, $b_k$ of real numbers such that $a^{k}+ b^{k} \to a+ b$ and $a \leq \liminf a^{k}$, $b\leq \liminf b^{k}$ we have  that $a^{k} \to a$ and $b^{k}\to b$, we get, by the lower semicontinuity of the three terms appearing in $F$,  that $\EE\left[\int_0^T (u^ k)^{\top}Nu^k\right]\to\EE\left[\int_0^T u^{\top}Nu\right] $ and so by expanding $\EE\left[\int_0^T (u^ k-u)^{\top}N(u^k-u)\right]$ and \textbf{(H5)} we conclude that $\|u^{k}\|_{2,2}\to \|u\|_{2.2}$. Therefore $u^{k}\to \bar{u}$ strongly in $\control$. Setting $z^{k}:= x^{k}- \bar{x}$ and $v^{k}= u^{k}-\bar{u}$, we have
 $$\ba{rcl}\dd z^{k}(t)&=& [A(t) z^{k}(t)+\delta^{k} A x^{k}+ B(t) v^{k}+ \delta^{k}B (t) u^{k}(t)+\delta^{k}e(t)]\dd t\\[4pt]
\; & \; &+ \sum\limits_{j=1}^d [C^{j}(t)z^{k}(t) +\delta^{k} C^{j}(t) x^{k}+D^{j}(t)v^{k}+\delta^{k}D^{j}(t)u^{k}(t)+\delta^{k} f^{j}(t)]\dd W^j(t),\\[4pt]
z^{k}(0) &=& \delta^{k}x_{0}.\ea
$$ 
 Since $v^{k}\to 0$ in $\control$, using the first estimate of \eqref{menqnoooasqwqwqw} and the fact that $(x^{k}, u^{k})$ is bounded in $\ito\times \control$, we obtain that $x^{k}\to x$ strongly in $(L^{2,\infty}_{\FF})^{n}$ and consequently, passing to the $(L^{2,2}_{\FF})^{n}$ limit in $x^{k}_{1}$ and $x_{2}^{k}$, also in $\ito$. Finally,  setting $\hat{p}^{k}:= p^{k}-\bar{p}$ and $\hat{q}^{k}:= q^{k}-\bar{q}$, we have that  \small
$$\ba{ll}
\dd \hat{p}^{k}(t) =&-[A(t)^{\top}\hat{p}^{k}(t) +\delta^{k}A(t) p^{k}(t)+ \sum_{j=1}^d [C^j(t)^{\top}(\hat{q})^{j}(t)+\delta^{k}C^j(t)^{\top}(q^{j})^{k}(t)]  + Q(t)z^{k}(t) ]\dd t \\[4pt]
\; & +\sum_{j=1}^d (\hat{q}^{j})^{k}(t) \dd W^j(t),\\[4pt]
 \hat{p}^{k}(T)=&Mz^{k}(T).\ea$$ \normalsize
Then,   applying the classical estimates for linear BSDEs (see e.g. \cite[ Chapter 7, Theorem 2.2]{YongZhou}) and using that $z^{k}(T)\to 0$ strongly in $(L^{2}_{\F_{T}})^{n}$, and that $(p^{k},q^{k})$  remain bounded in $\ito \times (L^{2,2}_{\FF})^{n\times d}$, we get that $(\hat{p}^{k}, \hat{q}^{k}) \to (0,0)$ strongly in $ (L^{2,2}_{\FF})^{n}\times (L^{2,2}_{\FF})^{n\times d}$. By passing to the limit in $\hat{p}^{k}_{1}$ and $\hat{p}^{k}_{2}$ we obtain the desired result.
\end{proof}
\vspace{0.3cm}

Define now the value function $v:\mathcal{P}_3\mapsto \RR$ of the $(P_{3,P})$ as a function of the parameters. Note that, under {\bf(H5)}, $v$ is well defined.  With the previous proposition, we can prove the following sensitivity result:
\begin{theorem}\label{teoremalq} Suppose that  {\bf(H5)} holds. Then,  $v$ is of class $C^{1}$. Moreover,  at any $P=(x_0,A,B,\{C^j\},\{D^j\},e,\{f^j\})\in\mathcal{P}_3$ the   directional derivative is given by  \small
\be\label{expresionderdireccionalcasolq}
\ba{lll}
Dv(P; \Delta P) &=& 
\bar{p}(0)\Delta x_0 +\EE\left(\int_0^T \bar{p}(t)^{\top}\left[\Delta A(t) \bar{x}(t) + \Delta B(t) \bar{u}(t) +\Delta e(t) \right]\dd t\right ) \\[4pt]
\; &\; & +\EE\left(\int_0^T  \sum_{j=1}^{d} \bar{q}^j(t)^{\top} \left[\Delta C^j(t)\bar{x}(t)+\Delta D^j(t)\bar{u}(t)+\Delta f^j(t)\right]  \dd t\right),
\ea
\ee\normalsize
where $\Delta P:=(\Delta x_0,\Delta A, \Delta B,\{ \Delta C^j\},\{\Delta D^j\},\Delta e,\{\Delta f^j\})$ and $(\bar{x}, \bar{u}, \bar{p}, \bar{q})= (x[P], u[P], p[P], q[P])$.
\end{theorem} 

\begin{proof}  The   Hadamard differentiability property for $v$  and expression \eqref{expresionderdireccionalcasolq} follow from  the surjectivity result in Lemma \ref{regularidadLQ}, the strong stability of the solutions proved in Proposition \ref{convergencia}, the identification of the Lagrange multipliers with the weak-Pontryagin multipliers proved in  Theorem \ref{Teoidentificacion}   and   \cite[Theorem 4.24]{BonSha},  dealing with   sensitivity results for the optimal value in  optimization problems in Banach spaces. Moreover, using again Proposition  \ref{convergencia} and expression \eqref{expresionderdireccionalcasolq} we easily check that $Dv(\cdot):  {\P}_{3}  \to  L(\P_3, \RR)$ is continuous, which implies the $C^{1}$ property.
\end{proof}
\begin{remark}\label{remarkutil} {\rm(i)} Note that if the nominal problem is deterministic then 
$$
Dv(P; \Delta P) =
\bar{p}(0)\Delta x_0 +\int_0^T \bar{p}(t)^{\top}\left[\EE\left(\Delta A(t)\right) \bar{x}(t) +\EE\left( \Delta B(t)\right ) \bar{u}(t) + \EE[\Delta e(t)] \right]\dd t 
$$
Therefore, the first order term of $v(P+\Delta P)-v(P)$ can be computed with the help of a deterministic differential  Riccati equation.  This could be useful in practice, since it provides a first order approximation for the value $v(P+\Delta P)$ of the stochastic LQ problem, whose solution is typically characterized in terms of Riccati backward stochastic differential equations, which are more difficult to solve than their deterministic counterpart.  \smallskip\\
{\rm(ii)} It could be interesting to study the extension of the above result for the case of indefinite control weight costs, i.e. when $N$ is not necessarily definite positive {\rm(}see {\rm \cite{ChenLiZhou98}, \cite[Chapter 6]{YongZhou}} and references therein{\rm)}.
\end{remark}

\subsection{Mean-Variance Portfolio Selection}\label{sens3}

Suppose that a market consists of $d+1$ assets  $S^{0}, S^{1}, \hdots, S^{d}$   whose prices are defined by
\be\label{mwemwmdnndndnsdsdnssssssaaaa}
\ba{rcl}
\dd S^{0}(t)&=& r S^{0}(t), \; \; \mbox{for } t \in [0,T], \; \; S^{0}(0)=1, \\[4pt]
dS(t)&=&\mbox{diag}(S(t))\mu(t)\dd t+\mbox{diag}(S(t))\sigma(t)\dd W(t) \; \; \mbox{for } \; t \in [0,T], \; S(0)= S_{0} \in \RR^{d},\\[4pt]
\ea\ee
where  $S:=(S^{1}, \hdots, S^{d})$ and for $a\in \RR^{d}$ the matrix $\mbox{diag}(a) \in \RR^{d\times d}$ is defined as $\mbox{diag}(a)^{ij}=\delta^{ij} ai$ for all $i,j \in \{1, \hdots, d\}$ ($\delta_{ij} $ is the Kronecker symbol). The precise properties on the processes $r\in L^{\infty}([0,T]; \RR)$,  $\mu \in (L^{\infty, \infty}_{\FF})^{d}$ and $\sigma \in (L^{\infty, \infty}_{\FF})^{d\times d}$ shall be given shortly and  will imply that  the financial market  is arbitrage-free and complete (see e.g. \cite[Chapter 1,Theorem 4.2 and 6.6]{KaraShreveFinance}). 

Given an initial wealth $x\in \RR$ and a  {\it self-financing portfolio} $\pi \in (L^{2,2}_{\FF})^{d}$ measured in units of wealth, the associated {\it wealth process} $X$ is defined through the SDE:
\be\label{ecXr}\ba{rcl}
\dd X(t)&=& \{ r(t)X(t)+\pi(t)^{\top}(\mu(t)-r(t){\bf 1}) \} \dd t+\pi(t)^{\top} \sigma(t) \dd W(t) \hspace{0.2cm} \mbox{for all } \; t \in [0,T],\\[4pt]
X(0)&=&x. \ea
\ee
where ${\bf 1}$ denotes the vector of ones in $\RR^{d}$.
For $A \in \RR$ we consider the problem (see e.g. \cite{duffie1991,ZhouLi00,Oks04}):
$$ 
\inf\limits_{(X,\pi) \in \I^{1} \times (L^{2,2}_{\FF})^{d}} \EE\left( \left[X-A\right]^{2} \right), \; \; 
\mbox{such that   \eqref{ecXr} is verified and } \; \; \EE\left( X(T) \right)=A.  \eqno(MVP)
$$
We then see then that the aim is  to minimize the risk (variance) subject to a guaranteed mean-return at the final time $T$.

We intend to compute the sensitivities of this problem with respect to its parameters. We thus define as \textit{parameter space} $\mathcal{P}_4:=\RR\times L^{\infty}([0,T])\times\RR\times (L^{\infty,\infty}_{\FF})^{d} \times (L^{\infty,\infty}_{\FF})^{d\times d}$. We will further say that $P=(x,r,A,\mu,\sigma)$ belongs to $\hat{\P}_4$ if $P\in \P_4$, $\sigma\sigma^{\top}\succeq \delta I_{d\times d}$ for some $\delta >0$, and
\be\label{mwrmwmnndndndsss}
\sum_{i=1}^{d} \left | \EE\left(\int_0^T [\mu_i(t) -r(t)]\dd t\right)\right	| >0.\ee
Note that $\hat{\P}_4$ is an open subset of $ \P_4$. Let us call $v(P):=\mbox{\textit{ value of }}(MVP)$, the corresponding optimal value function (as a function of the model parameters). On a first step we prove some estimates relating the norms of the portfolio and wealth.  As in the LQ-case, we compute the constants rather explicitly to show that they will not explode when we vary the model parameters.  

\begin{lemma}
\label{cotas} If $P=(x,r,A,\mu,\sigma)\in \hat{\P}_4$ and  $X$ satisfies \eqref{ecXr}, then \small
\begin{eqnarray}
\|\pi\|_{2,2}^{2}  &\leq & \frac{2}{\delta} \EE\left[ X(T)^2 \right]\left( 1+ 2 T\left ( \| r \|_{\infty} + \| \sigma^{-1} \{\mu-r{\bf 1}\}\|_{\infty}^{2} \right ) e^{2\left ( \| r \|_{\infty} + \| \sigma^{-1} \{\mu-r{\bf 1}\}\|_{\infty}^{2} \right )T}\right), 
\label{mqemqmemqmeqssssaa}\\
\| X \|_{ 2,2}^2 &\leq & T\EE\left( |X(T)|^{2}\right)e^{2 \left ( \| r \|_{\infty} + \| \sigma^{-1} \{\mu-r{\bf 1}\}\|_{\infty}^{2} \right ) T}.
\label{estimX}
\end{eqnarray}
\normalsize
\end{lemma}
\begin{proof}  By classical results on SDEs (e.g.\cite[Theorem 2.1]{Bismut76a}) we have that $X\in L_{\FF}^{2,\infty}$. Let us set $Z= \sigma^{\top} \pi$. We have that 
$$X(t)= x+ \int_{0}^{t} \left [ r(s)X(s)+Z^{\top} \sigma^{-1}(s) \{\mu(s)-r(s){\bf 1}\}\right ] \dd s + \int_{0}^{t} Z^{\top} \dd W(s).$$
By Itô's formula we have that 
$$ |X(t)|^{2}= |X(T)|^{2}- 2 \int_{t}^{T} X(s) \dd X(s)- \int_{t}^{T} |Z(s)|^{2} \dd s.$$
Using Lemma \ref{memmnnfbbbbbrbrwww}  we have that $\int_0^{\cdot} X\pi^{\top} \sigma \dd W$ is a martingale, and so taking the expectation in the above expression and omitting the time arguments, we get: \small
$$\ba{rcl}\EE\left( |X(t)|^2+ \int_{t}^{T} |Z|^{2} \dd s \right) &=& \EE\left( |X(T)|^{2} - 2 \int_{t}^{T} r|X|^2\dd s  - 2 \int_{t}^{T} X Z^{\top}\sigma^{-1}\{ \mu-r{\bf 1}\} \dd s\right),\\[4pt]
										\;      &\leq & \EE\left( |X(T)|^{2}+ 2 \| r \|_{\infty} \int_{t}^{T} |X|^2\dd s  + 2 \| \sigma^{-1} \{\mu - r{\bf 1}\}\|_{\infty} \int_{t}^{T} |X| |Z| \dd s\right),\\[4pt]
										\;      &\leq & \EE\left( |X(T)|^{2} + 2 \left ( \| r \|_{\infty} + \| \sigma^{-1} \{\mu-r{\bf 1}\}\|_{\infty}^{2} \right )\int_{t}^{T}|X|^{2} \dd s + \half \int_{t}^{T} |Z|^{2} \dd s \right),\ea   $$ \normalsize
from which 
\be\label{qemqemqmnnddnaaaaaa} 
\EE\left( |X(t)|^2+ \half \int_{t}^{T} |Z|^{2} \dd s \right)\leq \EE\left( |X(T)|^{2} + 2 \left ( \| r \|_{\infty} + \| \sigma^{-1} \{\mu-r{\bf 1}\}\|_{\infty}^{2} \right ) \int_{t}^{T}|X|^{2} \dd s \right).
\ee
Since the above inequality implies that 
$$
\EE\left( |X(t)|^2 \right)\leq \EE\left( |X(T)|^{2} + 2 \left ( \| r \|_{\infty} + \| \sigma^{-1} \{\mu-r{\bf 1}\}\|_{\infty}^{2} \right ) \int_{t}^{T}|X|^{2} \dd s \right).
$$
by Gronwall's Lemma we obtain that 
\be
\label{mkmkmkmk}
\EE\left( |X(t)|^2 \right)\leq  \EE\left( |X(T)|^{2}\right)e^{2 \left ( \| r \|_{\infty} + \| \sigma^{-1} \{\mu-r{\bf 1}\}\|_{\infty}^{2} \right ) T}
\ee
and \eqref{mqemqmemqmeqssssaa} follows from  estimate \eqref{qemqemqmnnddnaaaaaa}, Fubini's  Theorem, the definition of $Z$ and the fact that $\sigma \sigma^{\top} \geq \delta I_{d\times d}$. Finally, estimate \eqref{estimX} is a consequence of \eqref{mkmkmkmk} and  Fubini's  Theorem.\end{proof}

\medskip

For $P\in \P_4$ let us write the dynamic constraint \eqref{ecXr} as $G(X,\pi,P)=0$ with $$G(X,\pi,P)= x + \int_0^{\cdot}\left[ r(t)X(t)+\pi(t)^{\top} \{\mu(t)-r(t){\bf 1}\} \right] \dd t  +\int_{0}^{\cdot}\pi(t)^{\top} \sigma(t)\dd W(t)- X(\cdot)$$ and further consider $\hat{G}(X,\pi,P)=(G(X,\pi,P),\EE[X(T)]-A)$. Let us prove first:

\begin{lemma} 
\label{calif}
The function $\hat{G}:\I^{1}\times (L^{2,2}_{\FF})^{d} \times \P_4 \mapsto \I^{1}\times\RR$ is continuously Fr\'echet  differentiable. Furthermore, if $P\in  \hat{\P}_4$,  then $D_{(x,\pi)}\hat{G}(X,\pi,P):\I^1\times  (L^{2,2}_{\FF})^{d} \mapsto \I^1\times\RR$ is onto. 
\end{lemma}

\begin{proof}   The  Fr\'echet  differentiability  of $\hat{G}$ can be proved following exactly the same lines of the proof in Lemma \ref{regularidadLQ} and using  that   
 the second component of $\hat{G}$ is a continuous linear functional.  For the surjectivity claim, suppose  that $P \in \hat{\P}_{4}$ and that we are given $Y\in\I^{1}$ and $\xi\in\RR$. Then we need to find $(Z, \nu)\in \I^{1}\times (L^{2,2}_{\FF})^{d}$ such that:
\be\label{qmemqnndndnbbbasa}\ba{rcl}
-Z(\cdot)+\int_{0}^{\cdot}\left[ rZ+ \nu^{\top}(\mu - r{\bf 1})\right] \dd t+ \int_{0}^{\cdot}\nu^{\top} \sigma \dd W(t) &=&Y_{0}+ \int_{0}^{\cdot} Y_{1} \dd t+ \int_{0}^{\cdot} Y_{2} \dd W(t),\\[4pt]
\EE[Z(T)]&=&\xi.\ea
\ee
Let $i \in \{1,\hdots, d\}$ be such that $\kappa:= \EE\left(\int_0^T [\mu_i(t) -r(t)]\right)\dd t \neq 0$. Then, consider the portfolio $\nu$ with $\nu^{j}=0$ for $j\neq i$ and 
$$ \nu^{i}(t):= \left( \frac{\xi +  e^{\int_{0}^{T}r(t) \dd t}\left[Y_0  +  \EE\left(\int_{0}^{T} e^{-\int_{0}^{t}r(s)\dd s} Y_{1}(t) \dd t\right)\right]}{  e^{\int_{0}^{T}r(t) \dd t} \kappa}\right)e^{\int_{0}^{t}r(s) \dd s}. $$
Then, defining  $Z\in \I^{1}$ as the solution of 
$$\ba{rcl}
dZ(t) &=& \left[r(t)Z(t) + \nu^{\top}(\mu- r {\bf 1})-Y^1(t)\right]\dd t+\left[  \nu^{\top}\sigma  -Y^2(t)\right ]\dd W(t), \; \; \mbox{for all $t\in [0,T]$}, \\[4pt]
 Z(0)&=& - Y_0,\ea
$$
we easily check that $(Z,\nu)$ satisfies \eqref{qmemqnndndnbbbasa}.
\end{proof}\\

We now show that problem $(MVP)$ is attained. From here onwards $P:=(x,r,A,\mu,\sigma)\in \hat{\P}_4$ will denote a tuple of (reference, nominal) parameters. We denote by $v(P)$ the value of $(MVP)$ under parameters $P$.
\begin{lemma}
\label{attaingen}
We have that $v(P)<\infty$, and further this value is attained at a unique feasible pair $(X[P],\pi[P])$. Moreover,  there exists a unique weak-Pontryagin multiplier $$(p[P],q[P], \lambda_{E}[P])\in\I\times (L^{2,2}_{\FF})^{1\times d}\times \RR$$ satisfying:  
\be\label{dmwmwnrnwrnbbbssasas}
\ba{rcl}
\dd p[P](t) &=& -r(t)p[P](t)\dd t + q[P](t) \dd W(t) \; \; \; \mbox{ {\rm for all} $t \in ]0,T[$,}\\[4pt]
p[P](T) &=& 2[X[P](T)-A]+\lambda_{E}[P] \; \; \; \mbox{a.s. in $\Omega$}, \\[4pt]
p[P](t,\omega)(\mu(t,\omega)-r(t){\bf 1}) &=& -  \sigma(t,\omega) (q[P](t,\omega))^{\top} \; \; \;  \mbox{ {\rm a.s. in} $ [0,T]\times \Omega$.} 
\ea
\ee 
\end{lemma}

\begin{proof}
For the finiteness of $v(P)$ it suffices to prove that the feasible set is non-empty. Indeed,   by \eqref{mwrmwmnndndndsss} there is an $i$ such that $\EE[\int_0^T(\mu^i(t) -r(t))\dd t]\neq 0$. Therefore,   as in the proof of Lemma \ref{calif}, we may build the portfolio $\pi$ having $0$ in every coordinate except for the $i$-th one, which is set to
$$\left(\frac{A\exp\{-\int_0^T r(t)\dd t\}-x}{\EE[\int_0^T(\mu^i(t) -r(t))\dd t]}\right)e^{\int_{0}^{\cdot} r(t) \dd t}.$$
We easily see that   the corresponding wealth process has expected return equal to $A$ at time $T$ and so it is feasible. Suppose now that $(X^1,\pi^1)$ and $(X^2,\pi^{2})$ attain $v(P)$. This implies that $\EE\left [ (X^1(T))^2\right ]= \EE\left [ (X^2(T))^2\right ]$. If $X^1(T)$ were not almost surely equal to $X^2(T)$, by strict convexity of $Z \in L^{2}_{\F_{T}} \mapsto \EE[Z^2]$ we would get that the pair $\frac{1}{2}(X^1+X^2,\pi^1+\pi^{2} )$ is feasible and induces a strictly smaller value of the objective function, yielding a contradiction. Calling now 
$$\hat{X}(\cdot):= X^1(\cdot)-X^2(\cdot)=\int_0^{\cdot}\{ r(X^1-X^2 ) +(\pi^1-\pi^2)^{\top}( \mu- r {\bf 1}) \}\dd t + \int_0^{\cdot}(\pi^1-\pi^2)^{\top} \sigma \dd W(t),$$
we see that $\hat{X}(T)=0$  and from Lemma \ref{cotas}    that $\pi^1-\pi^2 \equiv 0$ and thus $\hat{X}(\cdot)\equiv 0$, and so that $X^1$ and $X^2$ are indistinguishable.
For attainability, suppose first that $(X^{k},\pi^{k})$ is a feasible optimizing sequence. We then know that $\EE\left ([X^{k}(T)]^2\right)$ is bounded.
By Lemma \ref{cotas} we get that $\pi^{k}$  is bounded in $(L^{2,2}_{\FF})^d$ and $X^{k}$ is bounded in $L^{2,2}_{\FF}$. Therefore, there exist $\pi \in (L^{2,2}_{\FF})^d$, $\hat{X} \in L^{2,2}_{\FF}$ such that, up to some subsequence, $(X^{k},\pi^{k})$ converges weakly to $(\hat{X}, \pi)$ in  $ L^{2,2}_{\FF}\times (L^{2,2}_{\FF})^d$. Moreover, since in $L^{2,2}_{\FF}$, we have that   $X_{1}^{k}$ converges weakly to $r \hat{X} + \pi^{\top}(\mu- r {\bf 1})$ and $X_{2}^{k}$ converges weakly to $\pi^{\top} \sigma$, we obtain that $X^{k}$ converges weakly in $\I^{1}$ to 
$$X(\cdot):=x+ \int_{0}^{\cdot} \left[ r \hat{X} + \pi^{\top}(\mu- r {\bf 1}) \right] \dd t+  \int_{0}^{\cdot} \pi^{\top} \sigma \dd W(t).$$
%
Therefore, using that $\I$ is injected continuously in $L^{2,2}_{\FF}$ by Proposition \ref{qmmqemqmndndndndaaa}{\rm(i)}, uniqueness of the weak limit implies that $\hat{X}=X$. Moreover, using  Proposition \ref{qmmqemqmndndndndaaa}{\rm(i)} again we see that  $\EE[X^{k}(T)]=A$ passes to the limit and we obtain that  $(X,\pi)$ is a feasible pair. Since the cost function is convex and strongly continuous we have that it is l.s.c. with respect to the weak convergence in $\I^{1}$, which implies that $(X,\pi)$ is the optimal pair.
%
Finally, the existence and uniqueness of the weak-Pontryagin multiplier $(p[P],q[P],\lambda_{E}[P])$ is a direct consequence of Theorem  \ref{madnandeueueuaasas}, Remark \ref{mmdnsnqwwwwwwwqeodoff}{\rm(i)} and  Lemma \ref{calif}.
Using \eqref{annqenbrbbrbsbsbsbsaaa}, it is straightforward to see that   $(p[P],q[P],\lambda_{E}[P])$   satisfies \eqref{dmwmwnrnwrnbbbssasas}.    
%
\end{proof}
\medskip


In order to simplify the sensitivity analysis, we use a change of variables that reduces the number of parameters. We let $X'(\cdot):=e^{-\int_0^{\cdot} r \dd t} X(\cdot)-Ae^{-\int_{0}^{T} r(t) \dd t}$  and for the portfolio variables we define the new ones by $\pi'(\cdot)= e^{-\int_0^\cdot r \dd s}\pi(\cdot)$.  With this change of variables, we easily see that for $P'=(x-Ae^{-\int_{0}^{T} r(t) \dd t},0,0,\mu-r{\bf 1},\sigma)$ we have the identity 
\be\label{memqemnnansaaaaaasanadad}v(P)=e^{2 \int_0^T r \dd s} v(P').\ee
Moreover,  $(\bar{X},\bar{\pi}, \bar{p},\bar{q}, \bar{\lambda}_E)=(X[P], \pi[P],p[P],q[P], \lambda_{E}[P])$ if and only if   
\be\label{qmmndnnnnnnsssasasa}\ba{rcl}
(X[P'], \pi[P'])&=&  (e^{-\int_0^{\cdot} r \dd t} \bar{X}(\cdot)-Ae^{-\int_{0}^{T} r(t) \dd t}, e^{-\int_0^{\cdot} r \dd t} \bar{\pi})\\[4pt]
(p[P'],q[P'], \lambda_{E}[P']) &=& \left(e^{\int_0^{\cdot} r \dd t-2 \int_{0}^{T}r \dd t} \bar{p}, \; e^{\int_0^{\cdot} r \dd t-2 \int_{0}^{T}r \dd t} \bar{q},  \; e^{-\int_0^{T} r \dd t}\bar{\lambda}_{E}\right). \ea
\ee

Therefore, in the following we will  consider general  perturbations with respect to the initial condition, the drift and diffusion coefficients, and for ease of notation we will write the value function only in terms of these parameters. That is, we shall assume that $r\equiv 0$, $A=0$ and consider perturbed parameters of the form $P(k):= (x^k,\mu^k,\sigma^k)$. In the end of this section we shall undo the above change of variables and analyse the full original problem.

 We will repeatedly use  the notation 
$$\ba{rcl}(X^k,\pi^k, p^k,q^{k}, \lambda^{k})&:=&(X[P({k})], \pi[P({k})],p[P({k})],q[P({k})], \lambda_{E}[P({k})]),\\[4pt]
		(\bar{X},\bar{\pi}, \bar{p},\bar{q}, \bar{\lambda}_E)&:=&(X[P], \pi[P],p[P],q[P], \lambda_{E}[P]),\ea
$$
 We now prove a stability result.
\begin{proposition}
\label{lemconv}
For any sequence $P(k)\rightarrow P$ we have that $v(P({k}))\rightarrow v(P)$ and further 
$$\left ( X^{k},\pi^{k}, p^{k},q^{k},\lambda^{k}\right )\rightarrow \left(\bar{X},\bar{\pi}, \bar{p},\bar{q}, \bar{\lambda}_E\right),$$
strongly in $\mathcal{I}^{1}\times (L^{2,2}_{\FF})^d \times \I^1 \times (L^{2,2}_{\FF})^{1,d}\times \RR$.
\end{proposition}

\begin{proof} First note that, since $P\in \hat{\P}_{4}$, there is a coordinate $i$  (which we fix) such that $\left[\EE(\int_0^T \mu^i(t)\dd t)\right]^2>0$. This implies that, for $k$ large enough, 
$$\left[\EE\left(\int_0^T (\mu^i)^{k}(t)\dd t \right)\right]^2\geq  \frac{1}{2}\left[\EE\left(\int_0^T \mu^i(t) \dd t\right)\right]^2>0$$
 and so the   portfolios with $i$-th component equal to $ -x^{k}/\EE(\int_0^T (\mu^i)^{k} \dd t)$ (and zero in the remaining ones) are feasible for $(MVP({k}))$. Using these feasible portfolios, we easily get the existence of  $K>0$ (independent of $k$) such that  
$v(P({k}))= \EE[X^{k}(T)^2]\leq K$ and thus by Lemma \ref{cotas} we obtain that $\pi^{k}$ is bounded in $(L^{2,2}_{\FF})^d$.
Now,  consider first only those $k$ such that $v(P({k}))\geq v(P)$ and define a portfolio $\nu^{k}$ equals to $\bar{\pi}$ except for the $i$-th coordinate where it equals $\bar{\pi}^i+z^{k}$, with 
$$z^{k}:=\frac{-x^{k}-\EE(\int_0^T \bar{\pi}^{\top}\mu^k \dd t)}{\EE(\int_0^T (\mu^i)^{k} \dd t)}.$$ 
Calling $Z^{k}(\cdot)=x^{k}+\int_0^{\cdot}(\nu^{k})^{\top}\mu^{k} \dd t+\int_{0}^{\cdot} (\nu^{k})^{\top} \sigma^{k}\dd W(t)$, we easily check that $(Z^{k}, \nu^{k})$ is feasible for  $(MVP({k}))$ and, since $z^{k} \to 0$, we have that $\EE [(Z^{k}(T))^2]\rightarrow \EE [(\bar{X}(T))^2]$. Hence,  for any $\epsilon >0$ and  $k$ large enough we obtain that $| v(P({k}))- v(P)|\leq v(P({k}))+\epsilon - \EE [(Z^{k}(T))^2]\leq \epsilon$. On the other hand, by considering  those $k$ such that $v(P({k}))\leq v(P)$, with a similar manner we can construct   out of $\pi^{k}$ a new portfolio $\xi^{k}$ obtained by modification of $\pi^{k}$'s $i$-th component in a way that it becomes feasible for the unperturbed problem.  More precisely, it suffices to set $(\xi^{j})^{k}= (\pi^{j})^{k}$ for $j\neq i$ and   $ (\xi^{i})^{k}:= (\pi^{i})^{k}+ \hat{z}^{k}$, where 
 $$\hat{z}^{k}:=\frac{-x-\EE(\int_0^T (\pi^{k})^{\top}\mu \dd t)}{\EE(\int_0^T \mu^i \dd t)}=\frac{x^k-x-\EE(\int_0^T (\pi^{k})^{\top}[\mu-\mu^k] \dd t)}{\EE(\int_0^T \mu^i \dd t)} .$$ 
Since $\pi^{k}$ is bounded in $(L^{2,2}_{\FF})^d$ we obtain that $\hat{z}^{k}\to 0$ and, as before, we get that for every  $\epsilon >0$ and $n$ large enough, $| v(P({k}))- v(P)|\leq \epsilon$, which proves convergence of the value functions.\\
Now let $ \pi$ be any weak limit point of $\pi^{k}$ in $(L^{2,2}_{\FF})^d$. Since, for $(y,Y,Z)\in \RR \times L_{\FF}^{2,2}\times  (L_{\FF}^{2,2})^{d}$ \small
 $$\langle X^{k},(y,Y,Z) \rangle_{\I} = x^{k}y+\EE\left(\int_0^T Y(\pi^{k})^{\top}\mu^{k} \dd t\right) +\EE\left( \int_0^T (\pi^{k})^{\top}\sigma^{k} Z \dd t \right),$$ \normalsize
 we get that, except for some subsequence, $\langle X^{k},(y,Y,Z) \rangle_{\I} \rightarrow \langle  X,(y,Y,Z) \rangle_{\I} $ ,
 where $  X(\cdot)=x+\int_0^{\cdot} \pi^{\top}\mu \dd t+\int_{0}^{\cdot}  \pi^{\top}\sigma \dd W(t) $, and thus $X^{k}\rightarrow   X$ weakly in $\I^1$.  Noticing that $\EE( X(T))=0$, and by virtue of convergence of the value functions, we have similarly as in Lemma \ref{attaingen} that  $(X,\pi)=(\bar{X},\bar{\pi})$. 
 By Proposition \ref{qmmqemqmndndndndaaa}{\rm(i)} we see that $X^{k}(T)$ converges weakly in $L^2_{\F_T}$ to $\bar{X}(T)$ and  using that $\EE( [X^{k}(T)]^{2}) \to \EE( [\bar{X}(T)]^{2})$   we obtain that $X^{k}(T)\to \bar{X}(T)$ strongly.
Let us write $\hat{X}^{k}=x+\int_0^{\cdot} (\pi^{k})^{\top} \mu \dd t+ \int_{0}^{\cdot}(\pi^{k})^{\top} \sigma \dd W(t)$. Then by Lemma \ref{cotas}:
$$\|\pi^{k}-\bar{\pi} \|_{2,2}^2\leq C\EE[(\hat{X}^{k}(T)-\bar{X}(T))^2],$$
where $C=C(\mu,\sigma)>0$ is some positive constant. Now, we have that $\EE[(\bar{X}(T)-X^{k}(T))^2]$ tends to zero, and 
$$\EE[(\hat{X}^{k}(T)-X^{k}(T))^2]\leq |x-x^{k}|^2+ T\|\pi^{k}\|_{2,2}^2\left[\|\mu-\mu^{k}\|_{\infty,\infty}^2+\|\sigma-\sigma^{k}\|_{\infty,\infty}^2\right],$$
which also tends to zero. We conclude with the triangle inequality that $\pi^{k}\rightarrow \bar{\pi}$ strongly in $(L^{2,2}_{\FF})^{d}$. Finally, since 
$$\|X^{k}-\bar{X}\|_{\I}^2= |x-x^{k}|^2+ \|(\pi^{k})^{\top}\mu^{k} - \bar{\pi}^{\top}\mu\|_{2,2}^2 + \|(\pi^{k})^{\top}\sigma^{k} - \bar{\pi}^{\top}\sigma\|_{2,2}^2,$$
we conclude that $X^{k}\rightarrow \bar{X}$ strongly in $\I^{1}$. Now, for the weak-Pontryagin multipliers $(p^{k},q^{k},\lambda^{k})\in \I^1 \times (L^{2,2}_{\F})^{1\times d}\times \RR$,  by \eqref{dmwmwnrnwrnbbbssasas} we have that:
$$\ba{ll}
\dd p^{k}  &=  q^{k}  \dd W(t) \; \;  \mbox{for all $t\in ]0,T[$}, \; \; \; \;  p^{k}(T) =  2X^{k}(T)+\lambda^{k}, \\[4pt]
0 &= p^{k}(t,\omega)\mu^{k}(t,\omega)+\sigma^{k}(t,\omega) (q^{k}(t,\omega))^{\top}, \; \; \; \mbox{for a.a. $(t,\omega) \in [0,T]\times \Omega$}.\ea
$$
We will show now that the $\lambda^{k}$ are bounded uniformly in $k$. Define $P^{k}=p^{k}-\lambda^{k}$. Then we know that $P^{k}(T)=2X^{k}(T)$ and $\dd P^{k}(t) = q^{k} \dd W(t)$. Since the $X^{k}(T)$ are $L^2_{\F_T}$-bounded, classical estimates for linear BSDEs  imply  that both $q^{k}$ and $P^{k}$ are bounded in $(L^{2,2}_{\FF})^{1\times d}$ and $L^{2,2}_{\FF}$, respectively. 
On the other hand we have that $(P^{k}+\lambda^{k})\mu^{k}+\sigma^{k} (q^{k})^{\top} = 0$, which proves that $\lambda^{k}\mu^{k}=-P^{k}\mu^{k}-\sigma^{k} (q^{k})^{\top} $ and thus: 
$$|\lambda^{k}|\|\mu^{k}\|_{2,2}=\|\lambda^{k}\mu^{k}\|_{2,2}\leq \|\mu^{k}\|_{\infty}\|P^{k}\|_{2,2}+ \|\sigma^{k}\|_{\infty}\|q^{k}\|_{2,2}.$$
The right hand-side of the above expression is uniformly bounded by the nature of the perturbations we have, and the estimates we already had. Further, we check that $ \|\mu^{k}\|_{2,2}$   is bounded away from zero since $\mu\not\equiv 0$ and thus 
$\lambda^{k}$ is bounded. Take now any subsequence of $(X^{k},\pi^{k},\lambda^{k})$. Then, there exists $\hat{\lambda} \in \RR$ such that, except for some subsequence,  $(X^{k},\pi^{k},\lambda^{k})$ converges strongly  to  $(\bar{X},\bar{\pi}, \hat{\lambda})$. This implies, by the classical estimates for linear BSDEs,  that the corresponding $(p^{k},q^{k})$ converge strongly in $L^{2,\infty}_{\FF} \times (L^{2,2}_{\FF})^{1\times d}$ to the solution $(p,q)$ of 
$$
\dd p(t)= q(t) \dd W(t) \; \; \mbox{for $t \in ]0,T[$}, \;  \; \; p(T) = \bar{X}(T)+\hat{\lambda}.
$$
 Further,  since $p^{k}(\cdot)= \lambda^{k} + \int_0^{\cdot}q^{k}\dd W$, we have that $p^{k}\to p$ strongly in $\I^{1}$. 
Moreover, since $(q^{k})^{\top}=-p^{k}(\sigma^{k})^{-1}\mu^{k}$ converges  in $(L^{2,2}_{\FF})^{1,d}$ to $-p\sigma^{-1}\mu$ we conclude that $p\mu+\sigma q^{\top} = 0$. Therefore, by the uniqueness of the weak-Pontryagin multiplier in Lemma \ref{attaingen}, we deduce that $(p,q,\hat{\lambda})=(\bar{p},\bar{q}, \bar{\lambda}_{E})$. This proves that the whole sequence $(p^{k},q^{k},\lambda^{k})$ converges to $(\bar{p}, \bar{q},\bar{\lambda}_{E})$ strongly in $\I^{1}\times (L^{2,2}_{\FF})^{1,d} \times \RR$.
\end{proof}\smallskip

\smallskip
By   Lemma \ref{calif}, Proposition \ref{lemconv} and arguing exactly as in the proof of Theorem \ref{teoremalq}, we have the following result:
\begin{proposition}
\label{sensMV}
The value function $v:\mathcal{P}_4\mapsto \RR$ is of class  $C^{1}$ on   $\hat{\mathcal{P}}_4$. Moreover, at every $P=(x,0,0,\mu,\sigma) \in \hat{\P}_{4}$ we have that
\be\label{msmamsaasasasssasadada} D_{(x,\mu,\sigma)}v(P; \Delta P) = \bar{p}(0)\Delta x+ \EE\left[\int_0^T \bar{\pi}(t)^{\top}\Delta\mu(t) \bar{p}(t)\dd t\right]+ \EE\left[\int_0^T \bar{\pi}(t)^{\top}\Delta\sigma(t) \bar{q}(t)^{\top}\dd t\right], 
\ee \normalsize
where $\Delta P=(\Delta x,0,0,\Delta \mu,\Delta \sigma)$ and   $(\bar{X},\bar{\pi}, \bar{p},\bar{q}, \bar{\lambda}_{E})= (X[P], \pi[P], p[P], q[P], \lambda_{E}[P])$ is given by Lemma \ref{attaingen}.
\end{proposition}
\medskip

We now unwind the change of variables done in order to reduce the size of the parameter space. In this way we obtain sensitivities with respect to the initial capital,  deterministic interest/saving rates, the desired return, the drift and the diffusion coefficients.

%
\begin{theorem}\label{corosens}
The value function $v:\mathcal{P}_4\mapsto \RR$ is  $C^{1}$ on  $\hat{\mathcal{P}}_4$. Moreover, at every  $P=(x,r,A,\mu,\sigma) \in \hat{\mathcal{P}}_4$  we have that \small
\be
\ba{l}
D_{x} v(P; \Delta x )= \bar{p}(0) \Delta x, \\[4pt]
D_{r} v(P;\Delta r)  = \EE\left( \int_{0}^{T}  \bar{p}(t)( \bar{X}(t)  - \bar{\pi}^{\top}  {\bf 1} )\Delta r    \dd t \right),\\[4pt]
D_{A}v(P;\Delta A)= - \bar{\lambda}_{E}\Delta A, \\[4pt]
 D_{\mu}v(P ; \Delta \mu)=  \EE\left[\int_0^T \bar{\pi}(t)^{\top}\Delta\mu(t) \bar{p}(t)\dd t\right],\\[4pt]
  D_{\sigma}v(P;\Delta \sigma)= \EE\left[\int_0^T \bar{\pi}(t)^{\top}\Delta\sigma(t) \bar{q}(t)^{\top}\dd t\right],
\ea
\ee\normalsize
where  $(\bar{X},\bar{\pi}, \bar{p},\bar{q}, \bar{\lambda}_{E})= (X[P], \pi[P], p[P], q[P], \lambda_{E}[P])$ is given by Lemma \ref{attaingen}.
\end{theorem}

\begin{proof}
Since $(x, r, A, \mu, \sigma) \to  (x-Ae^{-\int_{0}^{T} r(t) \dd t},0,0,\mu-r{\bf 1},\sigma)$  is  $C^{1}$,  we can apply the chain rule in \eqref{memqemnnansaaaaaasanadad}. Therefore, by \eqref{qmmndnnnnnnsssasasa} and Proposition \ref{sensMV}  we have that 
$$\ba{rcl}D_{x} v(P)&=& e^{2\int_{0}^{T} r(t) \dd t }\;  p[P'](0)= \bar{p}(0), \\[4pt]
	  D_{A} v(P)&=& e^{2\int_{0}^{T} r(t) \dd t}\; (-e^{-\int_{0}^{T} r(t) \dd t})  p[P'](0)= -e^{\int_{0}^{T} r(t) \dd t} \lambda_{E}[P']= -\bar{\lambda}_{E}, \\[4pt]
	     D_{\mu} v(P) \Delta \mu &=& e^{2\int_{0}^{T} r(t) \dd t }\;  \EE \left( \int_{0}^{T} e^{-\int_{0}^{t} r(s) \dd s} \bar{\pi}(t)^{\top} \Delta \mu (t) e^{\int_{0}^{t} r(s) \dd s-2\int_{0}^{T} r(s) \dd s} \bar{p}(t) \dd t\right),\\[4pt]
	     \; &= &   \EE\left[\int_0^T \bar{\pi}(t)^{\top}\Delta\mu(t) \bar{p}(t)\dd t\right], \\[4pt]
	      D_{\sigma} v(P) \Delta \sigma &=&  e^{2\int_{0}^{T} r(t) \dd t }\;  \EE \left( \int_{0}^{T} e^{-\int_{0}^{t} r(s) \dd s} \bar{\pi}(t)^{\top} \Delta \sigma (t) e^{\int_{0}^{t} r(s) \dd s-2\int_{0}^{T} r(s) \dd s} \bar{q}(t)^{\top} \dd t\right),\\[4pt]
	     \; &= &   \EE\left[\int_0^T \bar{\pi}(t)^{\top}\Delta\sigma (t) \bar{q}(t)^{\top}\dd t\right].
\ea
$$
Finally, setting $R(\cdot):= \int_{0}^{\cdot} \Delta r(t) \dd t$ and using that $p[P'](0)= \lambda_{E}[P']$, we obtain
\be\label{mqmennnaririririiri}\ba{rcl} D_{r} v( P; \Delta r)&=& 2 R(T) v(P) + e^{\int_{0}^{T} r(t) \dd t} p[P'](0) R(T)A \\[4pt]
					 \; & \; & -e^{2\int_{0}^{T} r(t) \dd t} \EE\left[ \int_{0}^{T} e^{-\int_{0}^{t} r(s) \dd s} \bar{\pi}^{\top} \Delta r(t) {\bf 1}e^{\int_{0}^{t} r(s) \dd s-2\int_{0}^{T} r(s) \dd s} \bar{p}(t) \dd t \right)\\[4pt]
					\; &=&  2 R(T) v(P) + \bar{\lambda}_{E} R(T)A- \EE\left(\int_{0}^{T} \bar{\pi}^{\top} \Delta r(t) {\bf 1} \bar{p}(t) \dd t \right).
					 \ea \ee
On the other hand, 
\be\label{qemqmemqdanda}\EE\left( \int_{0}^{T} \bar{X}(t)\bar{p}(t) \Delta r(t) \dd t\right)= \EE\left( R(T) \bar{X}(T)\bar{p}(T)- \int_{0}^{T} R(t) \dd (\bar{X}(t)\bar{p}(t))\right).\ee
By It\^o's formula, we can write 
$$\ba{rcl}  \dd (\bar{X}(t)\bar{p}(t))&=&  \left[r\bar{X}(t)p(t) + \bar{\pi}^{\top}( \mu(t) - r(t) {\bf 1}) \bar{p}(t)- \bar{X}(t) r(t) \bar{p}(t) + \bar{\pi}(t)^{\top} \sigma(t)\bar{q}(t)^{\top}  \right]\dd t \\[4pt]
						\;    & \; & +\left[ \bar{X}(t) \bar{q} + \bar{p}(t) \bar{\pi}(t)^{\top}\sigma(t) \right] \dd W(t).  \ea$$
Since, by the third line in \eqref{dmwmwnrnwrnbbbssasas},   $( \mu(t) - r(t) {\bf 1}) \bar{p}(t)= \sigma(t)\bar{q}(t)^{\top}$ we obtain with   Lemma \ref{memmnnfbbbbbrbrwww} that $\EE\left( \int_{0}^{T} R(t) \dd (\bar{X}(t)\bar{p}(t))\right)=0$. Therefore, by   \eqref{qemqmemqdanda} and the second line in \eqref{dmwmwnrnwrnbbbssasas}, we get \small
\be\label{qmemqnndnndnabbsassasasas}\ba{rcl}\EE\left( \int_{0}^{T} \bar{X}(t)\bar{p}(t) \Delta r(t) \dd t\right)=  \EE\left( R(T) \bar{X}(T)\bar{p}(T)\right)&=&R(T)\EE\left( \bar{X}(T)[2(\bar{X}(T)- A) + \bar{\lambda}_{E} ] \right), \\[4pt]
 \; &=& 2R(T) v(P) +  \bar{\lambda}_{E}R(T)A.
\ea
\ee\normalsize
The conclusion follows from \eqref{mqmennnaririririiri} and \eqref{qmemqnndnndnabbsassasasas}.
\end{proof}

\subsubsection{Comparison with a known explicit result}

We want to compare the theoretical sensitivities we obtain with those coming from a simplified model where an explicit solution is known. 
We choose to compare our results with the model in \cite[Example 4.1]{Oks04} (with null jump component).  More precisely, we consider the $(MVP)$ problem with $d=1$, 
$r\equiv 0$ and   $\mu(\cdot): [0,T] \to \RR$, $\sigma: [0,T] \to \RR$ being deterministic  bounded functions. Assuming that $\int_{0}^{T}\mu(t)\dd t \neq 0$ and that $\sigma$ is  uniformly positive,   problem $(MVP)$ can be explicitly solved (see \cite{Oks04} for the details).  Setting $\Sigma := \mu /\sigma$, the optimal portfolio and optimal states and adjoint states are given by 
\begin{eqnarray}
\bar{X}(t) &=& A+ \frac{x-A}{e^{\int_0^T \Sigma^2_s ds}-1}\left [ e^{\int_0^T \Sigma^2_s \dd s} e^{-\int_0^t \left(\Sigma_s  \dd W_s + \frac{3}{2}\Sigma^2_s  \dd s \right)} -1  \right ], \notag \\
\bar{\pi}(t) &=&  \frac{ (A-x)\mu(t) e^{\int_0^T \Sigma^2_s \dd s}}{\sigma(t)^2 \left ( e^{\int_0^T \Sigma^2_s \dd s}-1\right )}  e^{-\int_0^t \left(\Sigma_s  \dd W_s + \frac{3}{2}\Sigma^2_s \dd s \right)}, \label{piopt}\\
\bar{p}(t) &=& \frac{2(x-A)}{e^{\int_0^T \Sigma^2_s \dd s}-1}  e^{-\int_0^t \left(\Sigma_s \dd W_s + \frac{1}{2}\Sigma^2_s \dd s \right)},  \\
\bar{q}(t) &=& \frac{ 2(A-x)}{e^{\int_0^T \Sigma^2_s \dd s}-1}\Sigma(t) e^{-\int_0^t \left(\Sigma_s \dd W_s + \frac{1}{2}\Sigma^2_s \dd s \right)}.
\end{eqnarray} 
Thus, setting $P=(x,0,A,\mu,\sigma)$ and $\Delta P=(\Delta x,0,\Delta A,\Delta \mu,\Delta \sigma)$ from Theorem \ref{corosens} we have
\begin{equation}
Dv(P; \Delta P) = \bar{p}(0)[\Delta x -\Delta A]+\EE\left [\int_0^T \bar{p}(t)\bar{\pi}(t)\Delta \mu(t) \right ]+\EE\left [\int_0^T \bar{q}(t)\bar{\pi}(t)\Delta \sigma(t) \right ].\notag
\end{equation}
If we assume that  $\Delta \mu$ and $\Delta \sigma$ are deterministic, a brief computation then yields to:
\begin{eqnarray}
D_xv(P;\Delta x) &=& \frac{2(x-A)\Delta x}{ e^{\int_0^T \Sigma^2_s \dd s}-1}, \label{sensx}\\
D_Av(P;\Delta A) &=&  -\frac{2(x-A)\Delta A}{ e^{\int_0^T \Sigma^2_s \dd s}-1},\\
D_{\mu}v(P;\Delta \mu) &=& - \frac{2 (A-x)^2 e^{\int_0^T \Sigma^2_s \dd s}}{ \left (e^{\int_0^T \Sigma^2_s \dd s}-1\right )^2} \int_0^T \frac{\mu(t)\Delta \mu(t)}{\sigma(t)^2}\dd t, \\
D_{\sigma}v(P;\Delta \sigma) &=& \frac{2 (A-x)^2 e^{\int_0^T \Sigma^2_s \dd s}}{ \left (e^{\int_0^T \Sigma^2_s \dd s}-1\right )^2} \int_0^T \frac{\mu(t)^2\Delta \sigma(t)}{\sigma(t)^3}\dd t. \label{senssigma}
\end{eqnarray}
Since we know explicitly the solution,   we can actually  verify that 
$$v(P) = \frac{(x-A)^2}{e^{\int_0^T \Sigma^2_s \dd s}-1},$$
and thus computing its derivatives we easily recover  \eqref{sensx}-\eqref{senssigma}.

\bibliographystyle{plain}
\bibliography{projjf3}
\end{document}